\documentclass[12pt]{amsart}
\usepackage{amsmath,amssymb,amsbsy,amsfonts,latexsym,amsopn,amstext,cite,
                                               amsxtra,euscript,amscd,bm}
\usepackage{url}

\usepackage{mathrsfs}

\usepackage{color}
\usepackage[colorlinks,linkcolor=blue,anchorcolor=blue,citecolor=blue,backref=page]{hyperref}

\usepackage{graphics,epsfig}
\usepackage{graphicx}
\usepackage{float}
\usepackage{epstopdf}
\hypersetup{breaklinks=true}
\usepackage{verbatim}
\usepackage{graphicx}
\usepackage{caption}
\usepackage{subcaption}

\usepackage{soul} 
\usepackage{ulem}

\usepackage{float}
\usepackage{epstopdf}
\hypersetup{breaklinks=true}

\usepackage[np]{numprint}
\npdecimalsign{\ensuremath{.}}

\usepackage{bibentry}

\usepackage[english]{babel}
\usepackage{mathtools}
\usepackage{todonotes}
\usepackage{url}
\usepackage[norefs,nocites]{refcheck}

\begin{document}
\newtheorem{theorem}{Theorem}
\newtheorem*{thm*}{Theorem}
\newtheorem{lemma}[theorem]{Lemma}
\newtheorem{lem}[theorem]{Lemma}
\newtheorem{prop}[theorem]{Proposition}
\newtheorem{cor}[theorem]{Corollary}
\newtheorem*{conj}{Conjecture}
\newtheorem{question}[theorem]{Question}
\newtheorem{defn}[theorem]{Definition}
\newtheorem{rem}[theorem] {Remark} 



\numberwithin{equation}{section}
\numberwithin{theorem}{section}
\numberwithin{table}{section}

\newcommand{\rad}{\operatorname{rad}}

\newcommand{\Z}{{\mathbb Z}} 
\newcommand{\Q}{{\mathbb Q}}
\newcommand{\R}{{\mathbb R}}
\newcommand{\C}{{\mathbb C}}
\newcommand{\N}{{\mathbb N}}
\newcommand{\F}{{\mathbb F}}
\newcommand{\FF}{{\mathbb F}}
\newcommand{\fq}{\mathbb{F}_q}
\newcommand{\rmk}[1]{\footnote{{\bf Comment:} #1}}

\renewcommand{\mod}{\;\operatorname{mod}}
\newcommand{\ord}{\operatorname{ord}}
\newcommand{\TT}{\mathbb{T}} 
 
\renewcommand{\i}{{\mathrm{i}}}
\renewcommand{\d}{{\mathrm{d}}}
\renewcommand{\^}{\widehat}
\newcommand{\HH}{\mathbb H}
\newcommand{\Vol}{\operatorname{vol}}
\newcommand{\area}{\operatorname{area}}
\newcommand{\dist} {\operatorname{dist}}

\newcommand{\tr}{\operatorname{tr}}
\newcommand{\norm}{\mathcal N} 
\newcommand{\disc}{\operatorname{disc}}

\newcommand{\intinf}{\int_{-\infty}^\infty}

\newcommand{\ave}[1]{\left\langle#1\right\rangle} 
\newcommand{\Var}{\operatorname{Var}}
\newcommand{\Prob}{\operatorname{Prob}}
\newcommand{\Cov}{\operatorname{cov}}

\newcommand{\sym}{\operatorname{Sym}}
\newcommand{\Sym}{\operatorname{Sym}}

\newcommand{\CA}{{\mathcal C}_A}
\newcommand{\cond}{\operatorname{cond}} 
\newcommand{\lcm}{\operatorname{lcm}}
\newcommand{\Kl}{\operatorname{Kl}} 
\newcommand{\leg}[2]{\left( \frac{#1}{#2} \right)}  
\newcommand{\Li}{\operatorname{Li}}

\newcommand{\sumstar}{\sideset \and^{*} \to \sum}

\newcommand{\sumf}{\sum^\flat}
 
\newcommand{\conv}{*}

\newcommand{\Ht}{\operatorname{Ht}}

\newcommand{\E}{\operatorname{\mathbb E}} 
\newcommand{\sign}{\operatorname{sign}} 
\newcommand{\meas}{\operatorname{meas}} 
\newcommand{\length}{\operatorname{length}} 

\newcommand{\GL}{\operatorname{GL}}
\newcommand{\SL}{\operatorname{SL}}
\newcommand{\Sp}{\operatorname{Sp}} 
\newcommand{\USp}{\operatorname{USp}} 

\newcommand{\re}{\operatorname{Re}}
\newcommand{\im}{\operatorname{Im}}
\newcommand{\res}{\operatorname{Res}}
  
\newcommand{\diam}{\operatorname{diam}}
 
\newcommand{\fixme}[1]{\footnote{Fixme: #1}}
 
\newcommand{\orb}{\operatorname{Orb}}
\newcommand{\supp}{\operatorname{Supp}}

\newcommand{\OP}{\operatorname{Op}} 
\newcommand{\opS}{\operatorname{S}} 
\newcommand{\opT}{\operatorname{T}} 
\newcommand{\HN}{{\cH}_{N}}
\newcommand{\OPN}{\operatorname{Op}_N} 
\newcommand{\TN}{\opT_N} 
\newcommand{\TNr}{\opT_N^{(r)}} 
\newcommand{\UNr}{U_{N,r}} 
\newcommand{\TM}{\opT_M} 
\newcommand{\Hp}{\cH_p} 
\newcommand{\OPp}{\operatorname{Op}_p} 
\newcommand{\Tp}{\operatorname{T}_p} 
 \newcommand{\diag}{\operatorname{diag}}

 \newcommand{\Fp}{\mathbb F_p}
 \newcommand{\Gal}{\operatorname{Gal}} 

 
\newcommand{\commZ}[2][]{\todo[#1,color=blue!60]{Z: #2}}
\newcommand{\commI}[2][]{\todo[#1,color=green!60]{I: #2}}
\newcommand{\commII}[2][]{\todo[#1,color=magenta!60]{I: #2}}
\newcommand{\commP}[2][]{\todo[#1,color=red!60]{P: #2}}
\newcommand{\commA}[2][]{\todo[#1,color=yellow!60]{A: #2}}

\def\ccr#1{\textcolor{red}{#1}}
\def\cco#1{\textcolor{orange}{#1}}
\def\ccc#1{\textcolor{cyan}{#1}}

\def\Degf{e} 

\def\cont{\mathrm{cont}}

\newcommand{\bs}{\boldsymbol}
\def \a{\alpha} \def \b{\beta} \def \d{\delta} \def \e{\varepsilon} \def \g{\gamma} \def \k{\kappa} \def \l{\lambda} \def \s{\sigma} \def \t{\vartheta} \def \z{\zeta}

\newcommand{\mb}{\mathbb}

\def\sssum{\mathop{\sum\!\sum\!\sum}}
\def\ssum{\mathop{\sum\ldots \sum}}
\def\iint{\mathop{\int\ldots \int}}

\newcommand{\cc}[1]{{\color{magenta} #1}}

\newfont{\teneufm}{eufm10}
\newfont{\seveneufm}{eufm7}
\newfont{\fiveeufm}{eufm5}
%
%
\newfam\eufmfam
     \textfont\eufmfam=\teneufm
\scriptfont\eufmfam=\seveneufm
     \scriptscriptfont\eufmfam=\fiveeufm
%
%
\def\frak#1{{\fam\eufmfam\relax#1}}

\newcommand{\bflambda}{{\boldsymbol{\lambda}}}
\newcommand{\bfmu}{{\boldsymbol{\mu}}}
\newcommand{\bfxi}{{\boldsymbol{\eta}}}
\newcommand{\bfrho}{{\boldsymbol{\rho}}}

\def\eps{\varepsilon}

\def\fI{\mathfrak I}
\def\fK{\mathfrak K}
\def\fT{\mathfrak{T}}
\def\fL{\mathfrak L}
\def\fR{\mathfrak R}
\def\fQ{\mathfrak Q}

\def\fA{{\mathfrak A}}
\def\fB{{\mathfrak B}}
\def\fC{{\mathfrak C}}
\def\fM{{\mathfrak M}}
\def\fS{{\mathfrak  S}}
\def\fU{{\mathfrak U}}

\def\cPg{\cP_{\mathrm{good}}}

\def\cNg{\cN_{\mathrm{good}}}
\def\tcNg{\widetilde\cN_{\mathrm{good}}}
 
\def\sssum{\mathop{\sum\!\sum\!\sum}}
\def\ssum{\mathop{\sum\ldots \sum}}
\def\dsum{\mathop{\quad \sum \qquad \sum}}
\def\iint{\mathop{\int\ldots \int}}
 
\def\T {\mathsf {T}}
\def\Tor{\mathsf{T}_d}
\def\Tore{\widetilde{\mathrm{T}}_{d} }

\def\sM {\mathsf {M}}
\def\sL {\mathsf {L}}
\def\sK {\mathsf {K}}
\def\sP {\mathsf {P}}

\def\ss{\mathsf {s}}

\def \vk{\vec{k}}
\def \vl{\boldsymbol{\ell}}
\def \vm{\vec{m}}
\def \vn{\vec{n}}
\def \vx{\vec{x}}
\def \vf{\vec{f}}
\def \vu{\vec{u}}
\def \vv{\vec{v}}
\def \vw{\vec{w}}

\def \balpha{\bm{\alpha}}
\def \bbeta{\bm{\beta}}
\def \bgamma{\bm{\gamma}}
\def \bdelta{\bm{\delta}}
\def \bzeta{\bm{\zeta}}
\def \blambda{\bm{\lambda}}
\def \bchi{\bm{\chi}}
\def \bphi{\bm{\varphi}}
\def \bpsi{\bm{\psi}}
\def \bxi{\bm{\xi}}
\def \bnu{\bm{\nu}}
\def \bomega{\bm{\omega}}

\def \bell{\bm{\ell}}

\def\eqref#1{(\ref{#1})}

\def\vec#1{\mathbf{#1}}

\newcommand{\abs}[1]{\left| #1 \right|}

\def\Zq{\mathbb{Z}_q}
\def\Zqx{\mathbb{Z}_q^*}
\def\Zd{\mathbb{Z}_d}
\def\Zdx{\mathbb{Z}_d^*}
\def\Zf{\mathbb{Z}_f}
\def\Zfx{\mathbb{Z}_f^*}
\def\Zp{\mathbb{Z}_p}
\def\Zpx{\mathbb{Z}_p^*}
\def\cM{\mathcal M}
\def\cE{\mathcal E}
\def\cH{\mathcal H}

\def\le{\leqslant}
\def\leq{\leqslant}
\def\ge{\geqslant}
\def\leq{\leqslant}

\def\sfB{\mathsf {B}}
\def\sfC{\mathsf {C}}
\def\sfS{\mathsf {S}}
\def\sfI{\mathsf {I}}
\def\sfT{\mathsf {T}}
\def\L{\mathsf {L}}
\def\FF{\mathsf {F}}

\def\sE {\mathscr{E}}
\def\sS {\mathscr{S}}
\def\Gal {\mathrm{Gal}}

\def\cA{{\mathcal A}}
\def\cB{{\mathcal B}}
\def\cC{{\mathcal C}}
\def\cD{{\mathcal D}}
\def\cE{{\mathcal E}}
\def\cF{{\mathcal F}}
\def\cG{{\mathcal G}}
\def\cH{{\mathcal H}}
\def\cI{{\mathcal I}}
\def\cJ{{\mathcal J}}
\def\cK{{\mathcal K}}
\def\cL{{\mathcal L}}
\def\cM{{\mathcal M}}
\def\cN{{\mathcal N}}
\def\cO{{\mathcal O}}
\def\cP{{\mathcal P}}
\def\cQ{{\mathcal Q}}
\def\cR{{\mathcal R}}
\def\cS{{\mathcal S}}
\def\cT{{\mathcal T}}
\def\cU{{\mathcal U}}
\def\cV{{\mathcal V}}
\def\cW{{\mathcal W}}
\def\cX{{\mathcal X}}
\def\cY{{\mathcal Y}}
\def\cZ{{\mathcal Z}}
\newcommand{\rmod}[1]{\: \mbox{mod} \: #1}

\def\cg{{\mathcal g}}

\def\vy{\mathbf y}
\def\vr{\mathbf r}
\def\vx{\mathbf x}
\def\va{\mathbf a}
\def\vb{\mathbf b}
\def\vc{\mathbf c}
\def\ve{\mathbf e}
\def\vf{\mathbf f}
\def\vg{\mathbf g}
\def\vh{\mathbf h}
\def\vk{\mathbf k}
\def\vm{\mathbf m}
\def\vz{\mathbf z}
\def\vu{\mathbf u}
\def\vv{\mathbf v}

\def\e{{\mathbf{\,e}}}
\def\ep{{\mathbf{\,e}}_p}
\def\eq{{\mathbf{\,e}}_q}
\def\ek{{\mathbf{\,e}}_k}

\def\Tr{{\mathrm{Tr}}}
\def\Nm{{\mathrm{Nm}}}

\def\newr{t}

 \def\SS{{\mathbf{S}}}

\def\lcm{{\mathrm{lcm}}}

 \def\0{{\mathbf{0}}}

\def\({\left(}
\def\){\right)}
\def\l|{\left|}
\def\r|{\right|}
\def\fl#1{\left\lfloor#1\right\rfloor}
\def\rf#1{\left\lceil#1\right\rceil}
\def\sumstar#1{\mathop{\sum\vphantom|^{\!\!*}\,}_{#1}}

\def\mand{\qquad \mbox{and} \qquad}

\def\tblue#1{\begin{color}{blue}{{#1}}\end{color}}

\def\Fpd{\FF_{p^{2d}}}
 \newcommand{\fpone}{{\mathbb N}_{p^2}}
\newcommand{\fpp}{{\mathbb F}_{p^2}}

 \def \xbar{\overline x}

\newif\ifcomment

%

\let\varepsilon\varepsilon

\title[Quantum ergodicity for higher dimensional cat maps]{On quantum ergodicity for higher dimensional cat maps}

\author[P. Kurlberg]{P\"ar Kurlberg}
\address{Department of Mathematics, 
Royal Institute of Technology\\ SE-100 44 Stockholm, Sweden}
\email{kurlberg@math.kth.se}

\author[A. Ostafe] {Alina Ostafe}
\address{School of Mathematics and Statistics, University of New South Wales, Sydney NSW 2052, Australia}
\email{alina.ostafe@unsw.edu.au}

\author[Z. Rudnick] {Zeev Rudnick}
\address{School of Mathematical Sciences, Tel-Aviv University, Tel-Aviv 69978, Israel}
\email{rudnick@tauex.tau.ac.il}

\author[I. E. Shparlinski] {Igor E. Shparlinski}
\address{School of Mathematics and Statistics, University of New South Wales, Sydney NSW 2052, Australia}
\email{igor.shparlinski@unsw.edu.au}

 \dedicatory{Dedicated to the memory of Steven Zelditch} 
\begin{abstract}
We study eigenfunction localization for higher dimensional cat maps, a popular model of quantum chaos. These maps are given by  linear symplectic maps in $\Sp(2g,\Z)$, which we take to be ergodic. Under some natural assumptions, we show that there is a density one sequence of integers $N$ so that as $N$ tends to infinity along this sequence, all eigenfunctions of the quantized map at the inverse Planck constant $N$ are uniformly distributed. For the two-dimensional case ($g=1$), this was proved by P.~Kurlberg and Z.~Rudnick (2001). The higher dimensional case offers several new features and requires a completely different set of tools, including from additive combinatorics, such as  a bound of J.~Bourgain  (2005) for Mordell sums, and a study of tensor product structures for the cat map, which has never been exploited in this context.
\end{abstract}


\date{7 November, 2024}

\maketitle

\tableofcontents 

\section{Introduction}
\subsection{Quantum ergodicity and the quantized cat map}

Eigenfunction localization is one of the central topics of Quantum
Chaos. In this paper, we examine this question for an important ``toy
model'', the quantized cat map~\cite{HB}, aiming for higher
dimensional maps. Our techniques, after a preliminary reduction,
combine analytic number theory and additive combinatorics.

Denote by $\Sp(2g,\Z)$  the group of all   integer matrices $A$  which preserve the symplectic form  
\begin{equation}
\label{eq:omega}
\omega(\vec x,\vec y)  = \vec x_1 \cdot \vec y_2-\vec x_2\cdot \vec y_1,
\end{equation}
with $\vec x=(\vec x_1,\vec x_2)$, 
$\vec y=(\vec y_1,\vec y_2)\in \R^g\times \R^g$.  
Any $A\in \Sp(2g,\Z)$   generates a classical dynamical system via its action on the torus $\TT^{2g}= \R^{2g}/\Z^{2g}$. 
We say that this dynamical system is ergodic if for almost all initial positions $\vec x\in \TT^{2g}$, the orbit $\{A^j\vec x:~j\geq 0\}$ is uniformly distributed in $\TT^{2g}$. 
This is equivalent to $A$ having no eigenvalues which are roots of unity, see~\cite{Halm}.  

Associated to any $A\in \Sp(2g,\Z)$ is a quantum mechanical system. 
We briefly recall the key definitions: 
One   constructs for each integer $N\geq 1$ (the inverse Planck constant, necessarily an integer here) a Hilbert space of states $\HN = L^2((\Z/N\Z)^g)$ 
equipped with the  scalar product
\[
\langle \varphi_1,  \varphi_2 \rangle = \frac{1}{N^g} \sum_{\vec u \in (\Z/N\Z)^g}  \varphi_1(\vec u) \overline{ \varphi_2(\vec u)}, 
\qquad \varphi_1,  \varphi_2 \in \HN. 
\]

The basic observables are given by the unitary operators
\[
\TN(\vn):\HN\to \HN, \qquad \vn=(\vn_1,
\vn_2)\in \Z^g\times \Z^g=\Z^{2g},
\]
 as follows 
\begin{equation}
\label{eq:TN}
\left(\TN(\vn)\varphi \right)(\vec Q) = \e_{2N}(\vn_1\cdot \vn_2)\e_N(\vn_2\cdot \vec Q) \varphi (\vec Q+\vn_1),
\end{equation}
where hereafter we always follow the convention that integer arguments of functions on $\Z/N\Z$ are reduced modulo $N$ (that is, $ \varphi (\vec Q+\vn_1) =  \varphi (\vec Q+(\vn_1 \bmod N))$).
It is also easy to verify that~\eqref{eq:TN} implies  
\[
\TN(\vm) \TN(\vn)  = \e_{2N}\left(\omega\left(\vm,\vn\right)\right)\TN(\vm+\vn),
\]
where $\omega\left(\vm,\vn\right)$ is defined by~\eqref{eq:omega} and
\[
\e(z) =  \exp\(2\pi i z\),\quad \ek(z)=\e(z/k), 
\]
see also~\cite[Equation~(2.6)]{KR2001}.

For each real-valued function $f\in C^\infty(\TT^{2g})$ (an ``observable''), one associates a self-adjoint operator $\OPN(f)$ on $\HN$, analogous to a pseudo-differential operator with symbol $f$, defined by 
\begin{equation}
\label{eq:OpN}
\OP_{N}(f) = \sum_{\vn\in \Z^{2g}}\^f(\vn) \TN(\vn),
\end{equation}
where
\begin{equation}
\label{eq:f hat n}
f(\vec x)= \sum_{\vn\in \Z^{2g}} \hat f(\vec n)\e(\vec n \cdot \vec x).
\end{equation}

Assuming  $A=I\bmod 2$ (this condition can be weakened, see, for example, the definition 
of the subgroup $\Sp_\vartheta(2g, \Z)$ of  $\Sp(2g, \Z)$ as  in~\cite[p.~817]{KelmerAnnals}, 
with $d$ instead of $g$),
 for each value of the inverse Planck constant $N\geq 1$,  there is a
 unitary operator $U_N(A)$ on $\HN$, unique up to scalar multiples,
 which  generates the quantum evolution,   
 in the sense that for every observable  $f\in C^\infty(\TT^{2g})$, we have the exact Egorov property
\begin{equation}\label{egorov for T}
U_N(A)^*\OPN(f) U_N(A) = \OPN(f\circ A) , 
\end{equation}
where $U_N(A)^*= \overline{U_N(A)}^{\,t}$, 
we refer to~\cite{KRDuke, Rudnicksurvey} for  a detailed exposition in the case $g=1$ and~\cite{KelmerAnnals} for higher dimensions.  

 The stationary states of the system are the eigenfunctions of
 $U_N(A)$ and one of the main goals is to study their localization
 properties. In particular, given any normalized 
 sequence of eigenfunctions 
 $\psi_N\in \HN$, we ask if the expected values of observables  in these eigenfunctions converge, as $N\to \infty$, to the classical average 
 (see \S~\ref{sec:Observ} for precise definitions), that is, that
\begin{equation}
\label{eq:unif distr}
 \lim_{N\to \infty} \langle \OPN(f)\psi_N,\psi_N \rangle = \int_{\TT^{2g}} f(\vec x)d\vec x 
\end{equation}
for all $f\in C^\infty(\TT^{2g})$, 
   in which case we say that the sequence of eigenfunctions $\{\psi_N\}$ is uniformly distributed. 
   
   A fundamental result is the Quantum Ergodicity Theorem~\cite{Schnirelman, Zelditch, CdV}, valid in great generality, which
   in our setting asserts that if $A$ is ergodic, then for any
   {\bf orthonormal basis} $ \varPsi_N = \{\psi_{j,N}:~j=1,\ldots,N^g\}$ of
   eigenfunctions of $U_N(A)$ in $\HN$, there is a subset
   $\cS\subseteq\{1,\ldots,N^g\}$ with asymptotic density one (that is,
   $\sharp\cS/N^g\to 1$, where $\sharp \cS$ denotes the cardinality of $\cS$)
   so that $\psi_{j,N}$ are uniformly distributed for all $j\in \cS$,
   see~\cite{BdBv}. If \emph{all} eigenfunctions are uniformly
   distributed, the system is said to exhibit Quantum Unique
   Ergodicity~\cite{RudnickSarnakQUE}. In fact, more
   generally, 
  setting  
   \[
\Delta_A(f,N) = 
  \max_{\psi_N,\psi'_N }
\left| \langle \OPN(f)\psi_{ N}, \psi_{N}'  \rangle -\langle \psi_N,\psi_N'\rangle \int_{\TT^{2g}}
  f(\vx)d\vx  \right| ,
\] 
the maximum taken over all pairs of normalized eigenfunctions $\psi_N,\psi'_N$ of $U_N(A)$,
we ask if for all $f\in C^\infty(\TT^{2g})$, 
   \begin{equation}\label{QUE for all N's}
\lim_{\substack{N\to \infty \\ N \in \cN }}  \Delta_A(f,N)   = 0 , 
\end{equation} 
where $\cN$ is a set of integers of asymptotic density 1 (that is, 
$\# (\cN \cap [1,x]) = x + o(x)$ as $x \to \infty$).  
\begin{rem}
\label{rem:non-diag} 
It is
  interesting to note that even if we are mainly interested in scarring
  (that is, decay of diagonal matrix coefficients corresponding to $\psi_{N}' = \psi_{N}$ in the above
  definition of $\Delta_A(f,N)$ and establishing~\eqref{eq:unif distr}), for the full argument we still 
  need estimates for off-diagonals coefficients of the ``nontrivial''
  tensor component in \S~\ref{sec:eigenf-tens-prod}. 
\end{rem} 

 The two-dimensional ($g=1$) cat map is where the first counterexamples (``scars'') to QUE have been proved to exist~\cite{FND},  
  associated with  the $N$,  where the period $\ord(A,N)$ of the classical map reduced modulo $N$ was almost minimally small, 
  about  $2 \log N/\log \lambda$, where $\lambda>1$ is the largest eigenvalue of $A$.   We note that the relevance of the classical period to the quantum system was recognized early on in the theory~\cite{HB, Keating, dEGI}.
In~\cite{KR2001}, it was shown that if $\ord(A,N)$ was somewhat larger than $N^{1/2}$ (and $N$ satisfies a further genericity condition), then {\em all} eigenfunctions  in $\HN$ are uniformly distributed. 
 Note that the condition holds for almost all primes~\cite{EM}. 
Separately, it was shown that $\ord(A,N) $ is sufficiently large for almost all integers $N$. 

A breakthrough was made by Bourgain~\cite{BourgainGAFA}, who showed
that when $N=p$ is prime (that, and the prime power, cases are the
basic building block for the theory since the quantization with
respect to composite moduli arise as tensor products of quantizations
with respect to prime power moduli), for all eigenfunctions to be
uniformly distributed it suffices to take $\ord(A,p)>p^\varepsilon$,
for some $\varepsilon>0$, a condition that is much easier to establish
than a bound bigger than $p^{1/2}$.  This allowed Bourgain~\cite{BourgainGAFA}  to give a
polynomial rate of convergence for a version of~\eqref{QUE for all N's} 
over a sequence of almost all integers: for some
$\delta>0$, for almost all $N$ we have
$ \Delta_A(f,N) \le N^{-\delta} $.  Using a different approach, in
\cite{OSV} it is shown that one can take any $\delta< 1/60$.  


\subsection{Higher dimensional cat maps}
Higher dimensional cat maps offer several more challenges. 
In particular, we address the analogue of~\cite{KR2001}, namely all eigenfunctions in $\HN$  being uniformly distributed for almost all integers $N$.
We do not discuss other aspects of localizations, such as entropy
bounds~\cite{FN, Riviere} and showing that all semiclassical measures have
full support~\cite{Schwartz, DyatlovJezequel, Kim_etal}.   

In higher dimensions (that is, for $g>1$), there is a significant change. Kelmer~\cite{KelmerAnnals} has shown that if $A$ has nontrivial invariant rational istropic subspaces, then for all $N$, uniform distribution~\eqref{QUE for all N's} fails -- there are so-called scars. 

 So we assume that there are no nontrivial invariant rational isotropic subspaces. 
 We want to find a full density  sequence $\mathcal N$  of  integers $N$ for
which  {\em all} eigenfunctions of $U_N(A)$ are uniformly distributed,
that is,  if we fix  
$f\in C^\infty(\TT^{2g})$ then we have~\eqref{QUE for all N's}. 
If this holds for all $f$ then  we say that {\it $A$  satisfies  QUE for the
subsequence $\mathcal N$\/}.

Recall that we assume ergodicity,
equivalently,  that the eigenvalues of 
$A\in \Sp(2g,\Z)$ are not roots of unity.  For our results, we need to
impose a further condition on $A$, that no ratio of distinct eigenvalues is a root of unity. 
In addition, we assume that the characteristic polynomial $f_A(x)=\det(xI-A)\in \Z[x]$ is  separable (that is, has no multiple roots)

Our main result establishes~\eqref{QUE for all N's} for almost all integers under the above conditions on $A$:
\begin{theorem}\label{thm:integers}
  Let $A\in \Sp(2g,\Z)$, with a separable
  characteristic polynomial, be such that  no ratio of distinct eigenvalues is a root of unity. 
 Assume further that there are no nontrivial $A$-invariant rational isotropic subspaces. Then $A$ satisfies QUE as in~\eqref{QUE for all N's} 
  for some set $\cN$ of asymptotic density $1$. 
  \end{theorem}
  One can show  
  that if the
  characteristic polynomial of $A$ is irreducible, then there are no
  nontrivial $A$-invariant rational subspaces.  
  
   
\subsection{Plan of the proof}  
\label{sec:plan} 
We establish Theorem~\ref{thm:integers} via the following sequences of steps.

\begin{itemize}
\item[(i)]  To prove~\eqref{QUE for all N's}, it suffices to show it for the basic observables 
(translation operators) $\TN(\vn) = \OPN(f)$, $f(\vx)= \e\( \vx\cdot \vn \)$ 
 (see also~\eqref{eq:OpN}), with frequency $\vn$ growing slowly with $N$. 

  Assume that  the characteristic polynomial $f_A(x)=\det(xI-A)$ is
  irreducible over the rationals. Then we reduce the problem of
  estimating high powers $ \left| \langle \TN(\vn)\psi ,\psi' \rangle
  \right|^{4\nu}$ of the matrix elements  
   for all   normalized eigenfunctions $\psi, \psi'$, to a problem of estimating the
  number 
  of solutions to the matrix congruence   
\[
 A^{k_1}+\ldots +A^{k_{2\nu}} -A^{\ell_1} -\ldots -A^{\ell_{2\nu}}    \equiv  O\bmod N,  
\]
for the zero matrix $O$   
with $1\leq k_i,\ell_i \leq  \ord(A,N) $, $i =1, \ldots, 2\nu$, see
Lemmas~\ref{lem:basic inequality new} and~\ref{lin indep lemma} (since
indeed $f_A(x)$ being irreducible implies there is no nontrivial
zero-divisor in $\Q^{2g}$).

\item[(ii)]
In turn, this   number  can be treated by exponential sums. 
    However this reduction does not work directly due to the lack of
  nontrivial bounds on such sums except when $N=p$ is a prime,  
  modulo which the characteristic
  polynomial of $A$   splits completely, in which case we can apply a
  striking result of Bourgain~\cite{Bourgain Mordell} on short ``Mordell sums". The result, roughly speaking, is
  that there is some $\gamma>0$ so that for almost all split primes $p$,
\begin{equation}\label{eq:key bound on p intro}
       \max_{\psi , \psi'  }
  \left| \langle \opT_p (\vn)\psi ,\psi'  \rangle \right|
  \leq  p^{-\gamma}. 
\end{equation}

\item[(iii)]   To take advantage of the bound~\eqref{eq:key bound on p intro} for split primes, 
  we prove that the operators  $\TN(\vn)$ have a tensor
  product structure with respect to the
 Chinese  Remainder Theorem,  however with some losses depending on certain
 greatest common divisors.  Thus we  deal with the operators  $\Tp(\vn)$
 via exponential sums 
 and use the trivial bound 
 \[
  \left| \langle \TM(\vn)\psi ,\psi'  \rangle \right|\le 1, 
  \]
  where $M$ is the largest divisor of $N$ without split prime factors. 

 \item[(iv)] Finally, using some results from the {\it anatomy of  integers\/} (\S~\ref{sec:anatomy}) we show that for almost all
 integers $N$, the saving we obtain from the split primes  $p\mid N$, 
 exceeds the losses we incur 
in our version of the Chinese  Remainder Theorem. 

\item[(v)] When the characteristic polynomial $f_A(x)$ of $A$ is reducible,
  but separable,  we  require extra consideration, as the reduction to
  counting solutions of matrix congruences fails when $\vn$ is a
  non-trivial zero-divisor. 
    We make use of an additional tensor structure to
  reduce to the setting of congruences for a smaller dimensional case,
  see \S~\ref{sec:zero divisors} for details. 
\end{itemize}

 \subsection{Notation} 
 \label{sec:not}  
Throughout the paper, the 
notations 
\[X = O(Y),  \qquad X \ll Y, \qquad Y \gg X
\] 
are all equivalent to the
statement that the inequality $|X| \le c Y$ holds with some 
constant $c> 0$, which may depend on the matrix $A$, and occasionally, where obvious also on
the real parameter $\varepsilon$.

 We recall that the additive character with period $1$ is denoted by
\[
z \in \mathbb R\  \mapsto \ \e(z) =  \exp\(2\pi i z\).
\]
For an integer $k \ge 1$ it is also convenient to define 
\[
\ek(z) = \e(z/k).
\]

The letter $p$, with or without indices, always denotes prime numbers. 

Given an algebraic number $\gamma$ we denote by $\ord(\gamma, N)$ its order modulo $N$ 
(assuming that the ideals generated by $\gamma$ and $N$ are relatively prime in an appropriate number field). In particular, for an element $\lambda \in\F_{p^s}$, $\ord(\lambda,p)$ represents the order of $\lambda$ in $\F_{p^s}$. 

Similarly, we use $\ord(A,N)$ to denote the order of $A$ modulo $N$ (which always exists 
if $\gcd(\det A, N)= 1$ and in particular for $A\in  \Sp(2g,\Z)$). 

For a finite set $\cS$ we use $\sharp\cS$ to denote its cardinality.

 As usual, we say that a certain property holds for {\it almost all\/} elements of a sequence
 $s_n$, $n =1,2, \ldots$, if it fails for $o(x)$ terms with $n \le x$, as $x \to \infty$.
In particular, we say that it holds for {\it almost all\/}
 primes $p$ and positive integers $N$ 
 if  for $x \to \infty$, it fails for $o(x/\log x)$ primes $p\le x$ and $o(x)$ positive integers integers $N\le x$, respectively. 

Similarly,  we say that a certain property holds for {\it a positive proportion\/} of primes $p$  or, equivalently 
for a set of   {\it positive density\/},   if for some constant $c>0$,
which throughout this work may depend on the matrix $A$,  for
all sufficiently large $x$ it holds  
for at least $c x/\log x$ primes $p\le x$. 

\subsection{Acknowledgement} 

During the preparation of this work, 
P.K. was partially supported by the Swedish
  Research Council (2020-04036) and the Knut and Alice Wallenberg
  foundation (KAW 2019.0503). A.O. was supported by the Australian Research Council, DP230100530 and she gratefully acknowledges the hospitality and support of the Max Planck Institute for Mathematics in Bonn, where parts of this
work has been carried out.  Z.R. was supported by the European
  Research Council (ERC) under the European Union's Horizon 2020
  research and innovation programme (grant agreement No.~786758), and
  by the ISF-NSFC joint research program (grant No.~3109/23). 
  I.S. was supported by the Australian Research Council, DP230100534.

 \section{A Chinese  Remainder Theorem for the operators $\TN(\vn)$} 
\subsection{Observables} 
\label{sec:Observ}
We begin by defining the mixed translation operators. Given $r\in \Z$,
coprime
to $N$, and $\vn=(\vn_1, \vn_2)\in \Z^g\times \Z^g=\Z^{2g}$, we define
a unitary operator $\TN^{(r)}(\vn):\HN\to \HN$ by
\[
\left(\TN^{(r)}(\vn)\varphi \right)(\vec Q) = \e_{2N}(r \vn_1\cdot \vn_2)\e_N(r \vn_2\cdot \vec Q) \varphi (\vec Q+\vn_1) .
\]

We have
\begin{equation}
\label{eq:TNmn add}
\TN^{(r)}(\vm) \TN^{(r)}(\vn)  = \e_{2N}\left(r\omega\left(\vm,\vn\right)\right)\TN^{(r)}(\vm+\vn),
\end{equation}
where $\omega\left(\vm,\vn\right)$ is defined by~\eqref{eq:omega}. In particular, taking powers gives
\[
\left(  \TN^{(r)}\left(\vn \right) \right)^k = \TN^{(r)}(k \vn) .
\] 
The canonical commutation relations can be encapsulated in the relations
\[
\TN^{(r)}(\vn) \TN^{(r)}(\vm)=\e_N\left(r\omega\left(\vn,\vm\right) \right) 
\TN^{(r)}(\vm) \TN^{(r)}(\vn)  
\]
and
\[
(  \TN^{(r)}(\vn))^N =   \TN^{(r)}(N\vn) = (-1)^{r N\vn_1 \cdot \vn_2} I , \quad \vn=(\vn_1,\vn_2).
\]

For each  function $f\in C^\infty(\TT^{2g})$ on the classical phase space (an ``observable''), one associates an operator $\OP_{N,r}(f)$ on $\HN$, analogous to a pseudo-differential operator with symbol $f$, by 
\[
\OP_{N,r}(f) = \sum_{\vn\in \Z^{2g}}\^f(\vn) \TN^{(r)}(\vn),
\]
where $\^f(\vn)$ are defined by~\eqref{eq:f hat n}. 
If $f$ is real valued, then $\OP_{N,r}(f)$ is self-adjoint. 

When $r=1$,
we recover the definitions of $\TN = \TN^{(1)}$ and $\OPN = \OP_{N,1}$
in~\eqref{eq:TN} and~\eqref{eq:OpN}, respectively. 

Let $A\in \Sp(2g,\Z)$, satisfying the parity condition $A = I\bmod 2$. 
Fix $N\geq 1$ and $r$ coprime to $N$. Then there is a unitary operator $U_{N,r}(A):~\HN\to \HN$, unique up to a scalar multiple, so that we have the exact Egorov property
\begin{equation}\label{egorov for Tr}
U_{N,r}(A)^* \TN^{(r)}(\vn) U_{N,r}(A) =\TN^{(r)}(\vn A), 
\end{equation}
fo all $\vn\in \Z^{2g}$, which is a full analogue of~\eqref{egorov for T}.

\subsection{The Chinese Remainder Theorem and a tensor product
  structure}

Assume that the inverse Planck constant $N$ factors as $N=N_1 \cdot N_2$ with $N_1,N_2$ coprime. We then  use the Chinese Remainder Theorem $\iota: \Z/N\Z\cong \Z/N_1\Z \oplus \Z/N_2\Z$ to get an isomorphism
\[
\iota^*: L^2((\Z/N_1\Z)^g)\otimes  L^2((\Z/N_2\Z)^g) = \mathcal  \cH_{N_1} \otimes \mathcal  \cH_{N_2}\cong 
 \HN = L^2((\Z/N\Z)^g) \]
so that
\begin{equation}
\label{eq:iota tensor}
\iota^*(\varphi _1\otimes \varphi _2)(\vec Q) = \varphi _1(\vec Q \bmod N_1) \cdot  \varphi _2(\vec Q \bmod N_2) .
\end{equation}

The tensor product $\mathcal H_{N_1} \otimes \mathcal H_{N_2}$ carries the inner product 
\[
\|\varphi_1\otimes \varphi_2\| =\|\varphi_1\| \cdot \|\varphi_2\|
\]
and $\iota^*$  is actually an isometry, because it  maps the orthonormal basis of tensor products of normalized delta functions to normalized delta functions:
\[
\iota^* \left( N_1^{g/2}\delta_{\vec u} \otimes N_2^{g/2} \delta_{\vec v} \right) = N^{g/2} \delta_{\vec w}
\]
where $\vec w= \vec u \bmod N_1$, $\vec w = \vec v \bmod N_2$.

Assume $N=N_1 \cdot N_2$ with $N_1>1$, $N_2>1$ coprime.
Fix nonzero $r_1,r_2 \in \Z$ so that
\begin{equation}\label{equation with r1, r2}
  N_2 r_2 +  N_1 r_1=1.
\end{equation}
Necessarily $r_2$ is coprime to $N_1$ and $r_1$ is coprime to $N_2$. 
\begin{lemma}\label{lem:factorization of Tn}
For $\vn=(\vn_1, \vn_2)\in \Z^g\times \Z^g$,  
the mixed translation operator
$\TN(\vn)=\opT_N^{(1)}(\vn)$ is mapped, via the isomorphism
$\iota^*$,  
to $\opT_{N_1}^{(r_2)}(\vn)\otimes \opT_{N_2}^{(r_1)}(\vn)$:  
\[
\TN(\vn) \iota^*\left( \varphi _1\otimes \varphi _2 \right)  = \iota^*\( \left( \opT_{N_1}^{(r_2)}(\vn)\varphi _1 \right) \otimes  \left( \opT_{N_2}^{(r_1)}(\vn)\varphi _2 \right) \) .
\] 
\end{lemma}

\begin{proof}
Inserting~\eqref{equation with r1, r2} gives 
\begin{align*}
\e_{2N}(\vn_1\cdot \vn_2) & = \e\left( \frac{(N_2 r_2 +  N_1 r_1)\vn_1\cdot \vn_2}{2N_1N_2}\right)\\
& =  
\e\left( \frac{r_2 \vn_1\cdot \vn_2}{2N_1}\right) \e\left(  \frac{r_1 \vn_1\cdot \vn_2}{2N_2}\right) 
\end{align*} 
and for $\vy\in(\Z/N\Z)^g$, 
\begin{align*}
\e_N(\vn_2 \cdot \vy) & = \e\left(\frac{(N_2 r_2 +  N_1 r_1) \vn_2 \cdot \vy}{N_1N_2}  \right) \\
& = \e\left( \frac{r_2 \vn_2 \cdot \vy}{N_1} \right)\cdot \e\left( \frac{r_1 \vn_2 \cdot \vy}{N_2} \right) .
\end{align*} 

By definition~\eqref{eq:TN},
\[
\begin{split}
\{\TN(\vn)\iota^*(\varphi _1\otimes \varphi _2)\} (\vy)  = \e_{2N}(\vn_1\cdot \vn_2) \e_N(\vn_2 \cdot \vy) \iota^*(\varphi _1\otimes \varphi _2) (\vy+\vn_1)
\end{split}
\]
so that, using~\eqref{eq:iota tensor}, we obtain
\[
\begin{split}
\{\TN(\vn)\iota^*(\varphi _1\otimes \varphi _2)\} (\vy)   & = \e_{2N}(\vn_1\cdot \vn_2) \e_N(\vn_2 \cdot \vy) \iota^*(\varphi _1\otimes \varphi _2) (\vy+\vn_1) 
\\
&=  \e\left( \frac{r_2 \vn_1\cdot \vn_2}{2N_1}\right)e\left( \frac{r_2 \vn_2 \cdot \vy}{N_1} \right) \varphi _1(\vy+\vn_1) 
\\&\qquad  \cdot 
\e\left(  \frac{r_1 \vn_1\cdot \vn_2}{2N_2}\right) \e\left( \frac{r_1 \vn_2 \cdot \vy}{N_2} \right)\varphi _2(\vy+\vn_1) 
\\ &=  \left( \opT_{N_1}^{(r_2)}(\vn)\varphi _1 \right)(\vy) \cdot   \left( \opT_{N_2}^{(r_1)}(\vn)\varphi _2 \right)(\vy) 
\\
&=\iota^*\( \left(  \opT_{N_1}^{(r_2)}(\vn) \otimes \opT_{N_2}^{(r_1)}(\vn)  \right) \left( \varphi _1\otimes \varphi _2 \right)\) (\vy)
\end{split}
\] 
as claimed. 
\end{proof}

\subsection{Factorization of the quantized map}
We continue to assume a factorization of $N=N_1 \cdot N_2$ with $N_1>1$, $N_2>1$ coprime, and such that $r_1,r_2$ satisfy~\eqref{equation with r1, r2}. Then the Chinese Remainder Theorem induces an isometry  
\[
\HN \simeq \mathcal H_{N_1} \otimes \mathcal H_{N_2}, 
\]
which is respected by the translation operators (see Lemma~\ref{lem:factorization of Tn}). 
Furthermore, from now on, we identify the spaces 
$\HN$ and  $\mathcal H_{N_1} \otimes \mathcal H_{N_2}$ and thus we do not use the isomorphism  map $\iota^*$ anymore. 
We argue that we get a corresponding factorization 
of the quantized map $U_N(A) = U_{N,1}(A)$ defined by~\eqref{egorov for T} as a tensor product:

\begin{lem}
\label{eq:OperFact}
There is some $\zeta \in \C$, with $|\zeta|=1$, such that we have a
factorization
\[
U_{N}(A) = \zeta U_{N_1,r_2}(A) \otimes U_{N_2,r_1}(A).
\]
\end{lem}  

\begin{proof}
We saw in Lemma~\ref{lem:factorization of Tn} that 
\[
\TN(\vn A) = \opT_{N_1}^{(r_2)}(\vn A) \otimes \opT_{N_2}^{(r_1)}(\vn A) .
\]
By~\eqref{egorov for Tr}, we have
\[
\opT_{N_1}^{(r_2)}(\vn A)  = U_{N_1,r_2}(A)^* \opT_{N_1}^{(r_2)}(\vn) U_{N_1,r_2}(A) 
\]
and
\[
\opT_{N_2}^{(r_1)}(\vn A)  = U_{N_2,r_1}(A)^* \opT_{N_2}^{(r_1)}(\vn) U_{N_2,r_1}(A) .
\]
Hence $\widetilde U=  U_{N_1,r_2}(A)\otimes U_{N_2,r_1}(A)$ satisfies
\[
\begin{split}
\widetilde U^* \TN(\vn) \widetilde U &= 
\widetilde U^* \left( \opT_{N_1}^{(r_2)}(\vn) \otimes  \opT_{N_2}^{(r_1)}(\vn)  \right) \widetilde U 
\\
&=
\opT_{N_1}^{(r_2)}(\vn A)\otimes \opT_{N_2}^{(r_1)}(\vn A)=\TN(\vn A)
\end{split}
\]
for all $\vn\in \Z^{2g}$.  Since $U_N(A)$ is the unique (up to a
scalar multiple) unitary operator satisfying this relation, we must
have that $\widetilde U$ is a scalar multiple, of absolute value one,
of $U_N(A)$.
\end{proof}

\section{Bounding $\TN$ via the tensor product structure}
 \label{sec:tensor-product-yoga}
 
\subsection{Basic properties  of tensor products}
 We first recall a useful identity regarding operator norms of tensor
 products.  
  Let $V,W$ be finite dimensional inner product spaces, let
 $\opT_V :~V \to V$ and $\opT_W :~W \to W$ be linear maps, and let
 $\|\opT_V\|$ and $\|\opT_W\|$ denote the operator norms of $\opT_V$
 and $\opT_W$, respectively. There is a natural inner product on the
 tensor product $V \otimes W$ --- given orthonormal bases
 $\{v_{1},\ldots,v_{n}\}$ and $\{w_{1},\ldots, w_{m}\} $ for $V,W$,
 respectively, declare $\{ v_{i} \otimes w_{j}\}_{1 \le i \le n, 1 \le
   j \le m}$ to be an orthonormal basis for $V \otimes W$.
 We then have the relation 
\begin{equation}\label{eq: norn tensor prod}
\|\opT_V \otimes \opT_W\|= \|\opT_V\| \cdot \|\opT_W\|,
\end{equation}
between the operator
 norms (cf.~\cite[Page~299, Proposition]{ReSi}).

\subsection{Eigenspace decomposition in the coprime case}
\label{sec:eigensp-decomp-copr}
Assume that $\opT_V$ and $\opT_W$ are both diagonalizable, with eigenvalues being
roots of unity (in particular there exists, say minimal, integers
$t_{1},t_{2} > 0$ such that $\opT_V^{t_{1}} = I_V$ and $\opT_W^{t_{2}} = I_W$, where $I_V$ and  $I_W$ are the corresponding identity operators;  note
that we do not assume that $\opT_V$ and $\opT_W$ have the same dimensions).

The eigenspaces of $\opT_V \otimes \opT_W$ are
particularly easy to describe in terms of the eigenspaces of $\opT_V$ and $\opT_W$
when $\gcd(t_{1},t_{2})=1$. Namely, let $\{V_{i} \}_{i}$
denote the eigenspaces of $\opT_V$, and let $\{W_{j}\}_{j}$ denote the
eigenspaces of $\opT_W$; here we allow both $\opT_V,\opT_W$ to have eigenvalues
with multiplicities.
The eigenspaces of $\opT_V \otimes \opT_W$ are then given by
$\{ V_{i} \otimes W_{j} \}_{i,j}$. Further, if the eigenvalue
associated with $V_{i}$ is denoted $\mu_{i}$ and the eigenvalue
associated with $W_{j}$ is denoted by $\nu_{j}$, all eigenvalues of
$\opT_V \otimes \opT_W$ are of the form $\lambda_{i,j} = \mu_{i} \nu_{j}$.  In
particular, if $\gcd(t_{1},t_{2})=1$ we find that
$\mu_{i_{1}} \nu_{j_{1}} = \mu_{i_{2}} \nu_{j_{2}}$ implies that
$i_{1}=i_{2}$ and $j_{1}=j_{2}$. Informally, all multiplicities arise
by combining multiplicities from the $V_{i},W_{j}$ eigenspaces (note
that this is not true if $\gcd(t_{1},t_{2})>1$).

\subsection{Bounds using the tensor product structure}
\label{sec:bounds-using-tensor}
We now bound matrix coefficients of the special form
$\langle \TN(\vn) \psi, \psi'  \rangle$,
where, as in \S~\ref{sec:plan},
\[\TN(\vn)=  \OPN(\e\( \vx\cdot \vn \))
\]
and
$\psi,\psi'$ are eigenfunctions of $U_{N}(A)$.

Let $B$ be an element of $\Sp(2g,\Z)$. Assume that $N= N_1N_2$ with
  coprime integers $N_1>1$, $N_2>1$. 
Let $t_{i}$ denote the order of $B \mod {N_i}$, $i=1,2$. 
Further assume that 
\begin{equation}\label{eq: Coprime Ord}
\gcd(t_{1},t_{2})=1.
\end{equation}
Let $r_1,r_2\in \Z$ satisfy~\eqref{equation with r1, r2}. 
Taking  $A=B$ in Lemmas~\ref{lem:factorization of Tn} and~\ref{eq:OperFact}, 
we find
that
\[
\TN(\vn)= \opT_{N_1}^{(r_2)}(\vn) \otimes \opT_{N_2}^{(r_1)}(\vn)\quad \text{and}\quad U_{N}(B) = \zeta U_{N_1,r_{2}}(B)\otimes  U_{N_2,r_{1}}(B)
\] 
for some $\zeta\in\C^*$ with $|\zeta|=1$.

\begin{lem}
\label{lem:tensor-prod-bound}  
  Let $\psi, \psi'$ denote  norm one eigenfunctions of $U_{N}(B)$. With
  assumptions as 
  above, we then have
  \[
  |\langle \TN(\vn)\psi, \psi' \rangle|
  \le
  \max_{\varphi, \varphi' \in \Phi_{N_1, r_2}}
  | \langle  \opT_{N_1}^{(r_2)} (\vn)\varphi, \varphi' \rangle |, 
  \]   
  where $\Phi_{N_1, r_2}$ denotes the set of all
 eigenfunctions of  $U_{N_1,r_{2}}(B)$ of norm one.
\end{lem}

\begin{proof}
  Let $E,E'$ denote the eigenspaces of $U_{N}(B)$ containing
  $\psi, \psi'$, and let $\lambda, \lambda'$ denote the corresponding
  eigenvalues.

By the discussion in \S~\ref{sec:eigensp-decomp-copr}, we have
\begin{equation}\label{eq: decomp}
E = V_1 \otimes  V_2,
\end{equation}
where  $V_1$ and $V_2$ are  eigenspaces 
of
  $U_{N_1,r_{2}}(B)$ and $ U_{N_2,r_{1}}(B)$, 
respectively, and similarly we have
$E' = V_{1}' \otimes V_{2}'$. Thus, if we let
$S = P_{E'} \TN(\vn) P_{E}$, with $P_{E} : \HN \to E$ denoting the
orthogonal projection onto $E$ (and similarly for $P_{E'}$), we have
\begin{equation}\label{eq: T < S}
\max_{\substack{\psi \in E, \, \psi' \in E', \\ \|\psi\|=\|\psi'\|=1}}
|\langle \TN(\vn) \psi, \psi'  \rangle| =  \|S\|.
\end{equation}

Now, the decomposition~\eqref{eq: decomp} and Lemma~\ref{lem:factorization of Tn}, after a simple calculation, give that
\begin{equation}\label{eq: S S1 S2}
S = S_{1} \otimes S_{2}
\end{equation}
where $S_{1} = P_{V_1'} \opT_{N_1}^{(r_2)}(\vn) P_{V_1} $ and $S_{2} = P_{V_2'}
\opT_{N_2}^{(r_1)}(\vn) P_{V_2}$. 
Since both $S_{1},S_{2}$ arise as compositions of the unitary maps
$\opT_{N_1}^{(r_2)}(\vn),\opT_{N_2}^{(r_1)}(\vn)$ with orthogonal projections, they are both sub-unitary,
and we have the trivial bounds $\|S_{1}\|,\|S_{2}\| \le 1$.  
 Thus, by
the operator norm identity of tensor products~\eqref{eq: norn tensor prod}, 
 we see from~\eqref{eq: S S1 S2} that  
 \[\|S\| = \|S_{1}\| \cdot \|S_{2}\|\le \|S_{1}\|.
 \] 
  Using this,
 together with
\[
\|S_{1}\| =  \max_{\substack{\varphi \in V_1, \varphi' \in V_1' \\
\|\varphi\|=\|\varphi'\|=1}}
| \langle \opT_{N_1}^{(r_2)}(\vn) \varphi, \varphi' \rangle |, 
\]
and recalling~\eqref{eq: T < S}, the result now follows. \end{proof}

We now consider a more general case when instead of the coprimality condition~\eqref{eq: Coprime Ord}
we have 
\[
\gcd(t_{1},t_{2})=d.
\]
 
 Since any eigenfunction of  $U_{N}(A)$ is an
eigenfunction for $U_{N}\(A^d\)$ for any integer $d> 0$, we have
\begin{equation}\label{eq: A->A^d}
 \max_{\psi, \psi' \in \Psi_N}
| \langle \TN(\vn) \psi, \psi' \rangle|
\leq
 \max_{\widetilde \psi, \widetilde \psi' \in \widetilde \Psi_{N,d}}
| \langle \TN(\vn) \widetilde \psi, \widetilde \psi'\rangle|, 
\end{equation}
where $\Psi_N$ and  $ \widetilde \Psi_{N,d}$ denote the set of normalized
eigenfunctions  of $U_{N}(A)$ and $U_{N}(A^{d})$, respectively.

\section{Congruences and exponential sums} 
\subsection{Reduction to a counting problem}
\label{sec:Reduction to a counting problem}

 For a (row) vector $\vn\in \Z^{2g}$, $\vn\neq \0\bmod N$, we denote by $Q_{2\nu}(N;\vn)$ the number of solutions of the congruence 
 \begin{equation}\label{eq: Cong with-n}
\vn\left(A^{k_1}+\ldots +A^{k_{2\nu}} -A^{\ell_1} -\ldots -A^{\ell_{2\nu}}\right)   \equiv \0\bmod N,  
\end{equation}
with $1\leq k_i,\ell_i \leq  \ord(A,N) $, $i =1, \ldots, 2\nu$. 

The key inequality below connects the $4\nu$-th moment associated to
the basic observables $\TNr(\vn)$ 
with the number of
solutions to the system~\eqref{eq: Cong with-n}.  This kind of
inequality (for $\nu=1$) underlies the argument of~\cite{KR2001}, 
and also the argument of~\cite{BourgainGAFA}.

\begin{lem}\label{lem:basic inequality new}
Let $\0\neq \vn\in \Z^{2g}$ 
and let $r$ be an integer coprime to $N$.  
   Then
\begin{equation}\label{eq:basic inequality new}
   \max_{\psi,\psi'} 
  \left| \langle \TNr(\vn)\psi ,\psi'  \rangle \right|^{4\nu}
  \leq N^g \frac{Q_{2\nu}(N;\vn)}{\ord(A,N)^{4\nu} },
\end{equation}
where the maximum is taken over all pairs of normalized eigenfunctions of $U_{N,r}(A)$.
\end{lem}

\begin{proof} We  abbreviate $\tau=\ord(A,N)$.  Given a pair $(\psi,\psi')$  of normalized eigenfunctions of $U_{N,r}(A)$ with eigenvalues $\lambda,\lambda'$, put
  $\mu = \lambda'/\lambda$ and note that $\mu$ is a root of unity (since $\lambda$ and $\lambda'$ are). Let  
\[
  D(\vn) =   
  \frac 1{\tau}\sum_{i=1}^\tau
  \UNr(A)^{-i} \TNr(\vn) \UNr(A)^i \mu^{i}
  =
  \frac 1{\tau}\sum_{i=1}^\tau  \TNr(\vn  A^i) \mu^{i}
\]
be the $\mu$-twisted time averaged observable, where the last equality
comes 
from~\eqref{egorov for Tr}. 
Then for any pair of eigenfunctions
$(\psi, \psi')$   
of $\UNr(A)$, with eigenvalues $\lambda$ and $\lambda'$, we have
\[
  \langle \TNr(\vn)\psi,\psi' \rangle
  =\langle D(\vn)\psi,\psi' \rangle,
\]
see also the proof of~\cite[Proposition 4]{KR2001}. 
Put $H(\vn) = D(\vn)^{*}D(\vn)$; note that $H(\vn)$ is Hermitian.
Clearly
\[
|\langle D(\vn)\psi,\psi'  \rangle| \le \| D(\vn) \|
= \| H(\vn) \|^{1/2},
\]
where $\| H(\vn) \|$ denotes the operator norm of
$H(\vn)$.  Therefore for any $\nu\geq 1$,
\[
  | \langle \TNr(\vn)\psi,\psi' \rangle|^{4\nu}
  \leq \|H(\vn)\|^{2\nu}
  = \|H(\vn)^{\nu}\|^{2}.
\]

We bound the operator norm by the {\it Hilbert--Schmidt norm\/}
and obtain
\[
  \|H(\vn)^{\nu}\|^{2}\leq \|H(\vn)^{\nu}\|_{HS}^{2}
  = \tr((H(\vn)^{\nu})^{*}H(\vn)^{\nu})
  = \tr(H(\vn)^{2\nu}),
\] 
where $\tr$ denotes the operator trace. 
Finally, compute
\[
\begin{split}
 H(\vn)^{2\nu}&=
  \frac 1{\tau^{4\nu}}
  \sum_{k_1,\ldots,k_{2\nu},\ell_1\ldots,\ell_{2\nu}=1}^\tau \\
 & \qquad \qquad \qquad \times  \prod_{j=1}^{2\nu}
  \(\TNr\(\vn A^{k_j}\)  \TNr\(-\vn A^{\ell_j}\)\) \mu^{ \sum_{j=1}^{2\nu}(k_j-\ell_j)}
\\
&= \frac 1{\tau^{4\nu}} \sum_{k_1,\ldots,k_{2\nu},\ell_1\ldots,\ell_{2\nu}=1}^\tau  
\gamma(\vk,\vl,\mu{}) \TNr\( \vn \sum_{j=1}^{2\nu}  \( A^{k_j}- A^{\ell_j} \) \)
\end{split}
\]
with some complex coefficients $\gamma(\vk,  \vl, \mu)$, satisfying
$|\gamma(\vk,  \vl, \mu)|=1$, where $\vk = \(k_1,\ldots,k_{2\nu}\)$ and
$\vl = \(\ell_1\ldots,\ell_{2\nu}\)$, and where the last equality comes from~\eqref{eq:TNmn add}. Taking the  trace and using  
\[
  |\tr \TNr(\vm)| =
  \begin{cases}
    N^g & \text{if $\vec m=\0\bmod N$},
    \\0& \text{otherwise,}
  \end{cases}
\] 
we find
\[
  \left| \langle \TNr(\vn)\psi,\psi' \rangle\right|^{4\nu}
  \leq \tr\( H(\vn)^{2\nu}\)
  \leq \frac {N^g}{\tau^{4\nu}} Q_{2\nu}(N;\vn) 
\]
which concludes the proof. 
\end{proof}

\begin{rem}
\label{rem:strategy}
If the right hand side of~\eqref{eq:basic inequality new} tends to
zero, then~\eqref{QUE for all N's} is satisfied, that is, all
eigenfunctions of $U_N(A)$ are uniformly distributed, and more
generally, all off-diagonal matrix coefficients tend to zero.
Thus
Lemma~\ref{lem:basic inequality new} reduces the 
problem~\eqref{QUE for all N's} to a purely arithmetic issue.
\end{rem}

 \subsection{Linear independence of matrix powers} \label{sec:linear indep}
 
 The primal goal of this section is to show  that if the characteristic polynomial $f_A$ of $A$ is 
separable over $\Q$ we can essentially eliminate the dependence on the vector $\vn$ in our  argument, except in some special cases. 
In particular, instead of   $Q_{2\nu}(N;\vn)$  we can consider 
 the number of solutions of the congruence 
\begin{equation}\label{eq: Cong no-n}
A^{k_1}+\ldots +A^{k_{2\nu}}   \equiv A^{\ell_1} +\ldots +A^{\ell_{2\nu}}\bmod N,  
\end{equation}
with $1\leq k_i,\ell_i \leq  \ord(A,N) $, $i =1, \ldots, 2\nu$. 
This is based on the following result which is also used in our bounds on 
exponential sums. However, we first need to introduce the notion of zero-divisors
amongst the  row vectors $\vn\in \Z^{2g}$. 
For this  we first identify
$ \Q^{2g}\cong \Q[X]/(f_A(X))$  as a $ \Q[A]$ module. We say that $\vn$ is a 
{\it zero-divisor\/}, if its image $\widetilde \vn \in  \Q[X]/(f_A(X))$ is a zero-divisor
in this module (we follow the convention that zero is also a zero divisor, call all 
other zero divisors {\it nontrivial.\/})

\begin{lem}\label{lin indep lemma}
Let $A\in \Sp(2g,\Z)$ have a separable characteristic polynomial. Then
for any 
  row vector $\vn\in \Z^{2g}$, which is not
a zero-divisor, we have: 
\begin{itemize}
\item[(i)]  the vectors $\vn, \vn A,  \ldots,\vn A^{2g-1}$ are linearly independent;

\item[(ii)] there exists  some $p_0(A)$, depending only on $A$, such
  that for all primes $p>p_0(A) \|\vn\|_2^{2g}$, the vectors
  $\vn, \vn A, \ldots,\vn A^{2g-1}$ are linearly independent modulo
  $p$.
\end{itemize}
\end{lem}

\begin{proof}
In the case when the characteristic polynomial is irreducible,  Part~(i) is proved in~\cite[page~210]{KR2001} (for $n=2$) and in the proof of~\cite[Theorem 2.5]{OSV} (which is done over a finite field but it remains valid over any field). 

In our more general case of separability, assume that the vectors $\vn A^{i}$, $ i =0, \ldots, 2g-1$, are linearly dependent over $\Q$, that is,  there 
is a linear relation
\[
\sum_{i=0}^{2g-1} c_i \vn A^{i}=\vn \(\sum_{i=0}^{2g-1} c_i  A^{i}\) = \mathbf{0} 
\]
for some $c_i\in\Q$ not all zero. Since $\vn$ is not  a zero-divisor, we obtain
\[
\sum_{i=0}^{2g-1} c_i  A^{i} = \mathbf{0}.
\]
However, this shows that the minimal polynomial of $A$ has degree at most $2g-1$, which contradicts the fact that the minimal polynomial is $f_A$ since it is separable. This concludes the proof of Part~(i). 

To show
  Part~(ii) one considers the determinant of the matrix having rows
  $\vec n, \ldots,\vec n A^{2g-1}$ whose vanishing is equivalent to linear
  independence; it is an integer, nonzero by Part~(i), hence for all
  primes $p$ not dividing it we have linear independence mod $p$. 
\end{proof}

\subsection{Reduction to a system of exponential equations} 

We now consider~\eqref{eq: Cong with-n}  for a prime $N=p$.

Let $A\in \Sp(2g,\Z)$ have  separable characteristic polynomial
$f_A\in\Z[X]$.  Assume that $p$ is large enough so that $f_A$ is
separable modulo $p$, which also implies that $A$ is diagonalisable
over $\overline{{\mathbb F}}_{p}$.

Next, let 
\[
f_A(X)= h_1(X)\cdots h_\newr (X)  \bmod p
\]
be the factorization of $f_A$ into irreducible factors $h_i\in\F_p[X]$
of degrees $d_i=\deg h_i$, $i=1\ldots, \newr$. In
particular, any root of $h_i$ belongs to $\F_{p^{d_i}}$,
$i=1\ldots, \newr$.  For each $h_i$ we fix a root
$\lambda_i \in \F_{p^{d_i}}$ and consider the system of equations (in
the algebraic closure of $\F_p$)  
\begin{equation}
\label{eq:syst eigen}
\lambda_i^{k_1}+\cdots+\lambda_i^{k_{2\nu}}= \lambda_i^{\ell_1}+\cdots+\lambda_i^{\ell_{2\nu}}, \qquad i=1\ldots, \newr,
\end{equation}
with $1\leq k_j,\ell_j \leq\ord(A,p)$, $j= 1,\ldots, 2\nu$. 
It is easy to see that for all choices of roots $\lambda_1, \ldots, \lambda_\newr$ we get equivalent systems. 

Next,  we reduce counting the number  of solutions to~\eqref{eq: Cong no-n} to counting  the number  of solutions to~\eqref{eq:syst eigen}.

In fact our treatment depends only on the degrees $d_1, \ldots, d_\newr$ and so we 
denote the number of solutions to~\eqref{eq:syst eigen} by $R_{2\nu}(d_1,\ldots,d_\newr;p)$.

\begin{lem}
\label{lem:syst eigen}
Under the above assumptions,  there exists  some $p_0(A)$, depending only on $A$, such
that for any  
vector $\vn \in \Z^{2g}$, which is not a zero divisor,  and $p > p_0(A) \|\vn\|_2^{2g}$ 
we have 
\[
Q_{2\nu}(p;\vn) = R_{2\nu}(d_1,\ldots,d_\newr;p). 
\]
\end{lem}

 \begin{proof}
 Let us denote
 \[
 B=A^{k_1}+\dots +A^{k_{2\nu}} -A^{\ell_1} -\dots -A^{\ell_{2\nu}}.
 \]
 Multiplying~\eqref{eq: Cong with-n} by powers of $A$, we conclude that 
 \[
 \(\vn  A^i\)B \equiv \0 \bmod p,\qquad i=0,\ldots,2g-1,
 \]
 which is equivalent to 
 \[
 \begin{pmatrix}
 \vn \\ \vn  A\\ \cdots\\ \vn A^{2g-1}
 \end{pmatrix}B \equiv\0  \bmod p.
 \]
 From Lemma~\ref{lin indep lemma}~(ii), there  exists  some  $p_0(A)$, depending only on $A$,  such that for  $p > p_0(A) \|\vn\|_2^{2g}$ 
 all rows  $\vn, \vn A, \ldots,\vn A^{2g-1}$ are linearly independent modulo $p$, and thus from the above we conclude that $B$ vanishes over $\F_p$. 
  
 Since $A$ is diagonalisable over  $\overline{{\mathbb F}}_{p}$, the
 equation above is equivalent to 
 \[
 \Lambda^{k_1}+\ldots +\Lambda^{k_{2\nu}} \equiv \Lambda^{\ell_1} + \ldots +\Lambda^{\ell_{2\nu}}  \bmod p,
 \]
 where $\Lambda$ is a diagonal matrix with elements on the diagonal all the roots of $h_i$, $i=1,\ldots, \newr$. Since for each irreducible factor of $f$ modulo $p$, all roots are conjugate (that is, the roots of $h_i$ in $\F_{p^{d_i}}$ are $\lambda_i,\lambda_i^p,\ldots,\lambda_i^{p^{d_i-1}}$), we conclude the proof. \end{proof}

 \section{Multiplicative orders and exponential sums}
 
\subsection{Ergodicity and the order modulo $p$}
It is natural that our argument, as in~\cite{BourgainGAFA, KR2001, OSV},  rests 
on various results on multiplicative orders.

We begin by showing that the multiplicative orders of the eigenvalues
of $A\in \Sp(2g,\Z)$, and their ratios, are sufficiently large for
almost all primes.  The argument is a modification of that of
Hooley~\cite{Hool}.
  
We recall the definition of $\ord(\lambda,p)$ in \S~\ref{sec:not}
and also  that we  say that $p$ is split prime if the characteristic polynomial 
of the matrix $A$ splits completely modulo $p$. 

 \begin{lem}
 \label{lem:ord}
 Assume that $A\in \Sp(2g,\Z)$ has separable characteristic
 polynomial and that no eigenvalue or ratio of distinct eigenvalues is a root of unity. 
 Let
 $\lambda_1,\ldots,\lambda_{2g}$ be the   eigenvalues of $A$.
 Then for almost all  split primes $p$ we have
\[
 \ord(\lambda_i,p) , \ord(\lambda_i/\lambda_j ,p) > p^{1/2}/\log p,
 \quad 1\le  i\ne j \le 2g.    
 \]
 \end{lem}

  \begin{proof} 
  For a sufficiently large $Y\ge 2$, let 
 \[ A(Y) = 
 \prod_{n\leq Y}  \prod_{1 \le  i  \le 2g}   \Nm_{K/\Q}(\lambda_i^n -  1)  
  \prod_{1 \le j < h \le 2g}   \Nm_{K/\Q}(\lambda_j^n -  \lambda_h^n) , 
 \]
 where $\Nm_{K/\Q}\(\zeta\)$ is the  {\it norm\/} of  $\zeta \in K= Q\(\lambda_1, \ldots, \lambda_n\)$ in $\Q$. 
  Note that $A(Y) \neq 0$ because of the condition on the avoidance of roots of unity
  among the  eigenvalues and their ratios, and $A(Y) \in \Z$ since all eigenvalues are algebraic 
  integers.   Since  
  \[
   \Nm_{K/\Q}\(\lambda_i^n -  1\)=\prod_{\sigma\in \Gal(K/\Q)}\(\sigma(\lambda_i)^n -  1\)
\]
   and
\[
    \Nm_{K/\Q}\(\lambda_j^n -  \lambda_h^n\)=\prod_{\sigma\in \Gal(K/\Q)}\(\sigma(\lambda_j)^n-\sigma(\lambda_h)^n\),
\] 
  where both products are over all automorphisms $\sigma$ from the Galois group $\Gal(K/\Q)$ of $K$ over $\Q$,  and thus
\[
\log \Nm_{K/\Q}(\lambda_i^n -  1), \, \log \Nm_{K/\Q}(\lambda_j^n -  \lambda_h^n)\ll n,
\] 
we see that 
    \begin{equation}
\label{eq:large AY}
\log |A(Y)| \ll Y^2.
\end{equation} 
  
Let  $ \cP(Y)$ be the set of primes for which 
\[
 \min_{1 \le  i  \le 2g} \min_{1 \le j < h \le 2g}  \{\ord(\lambda_i,p) , \ord(\lambda_j/\lambda_h ,p)\} \le Y. 
 \]
We observe that for $p\in \cP(Y)$, we must have $p\mid A(Y)$, 
and hence 
\begin{equation}
\label{eq:omega AY}
\sharp  \cP(Y) \le \omega\(A(Y)\),
\end{equation}
where, as usual,  $\omega(k)$ denotes  the number 
of prime divisors of the integer $k \ge 1$. 
From the trivial observation that $\omega(k)! \le k$ and the Stirling formula, we derive 
\begin{equation}
\label{eq:omega k}
\omega(k) \ll \frac{\log k }{\log \log (k+2)}, \qquad k \ge 1.
\end{equation}
Putting together~\eqref{eq:large AY}, \eqref{eq:omega AY} and~\eqref{eq:omega k}, we see that 
\[
 \sharp\mathcal P(Y)\ll Y^2 /\log Y.
 \]
Since the number of primes $p\leq X$  is $\pi(X) \sim X/\log X$, we can take $Y=\sqrt{X}/\log X$ to assure that for  all but $o\(\pi(X)\)$ primes 
$p \le X$, we have 
\[
 \ord(\lambda_i,p) , \ord(\lambda_i/\lambda_j ,p)   >\sqrt{X}/\log X\geq \sqrt{p}/\log p,\quad  1 \le i \ne j \le 2g. 
 \]
Since splitting fields of polynomials are Galois extensions, by the Chebotarev Density Theorem, see~\cite[Theorem~21.2]{IwKow}, for a positive proportion of  primes $p$,  see our convention in \S~\ref{sec:not}, the characteristic polynomial of the matrix $A$ splits modulo $p$. This concludes the proof. 
\end{proof}

\subsection{Relation with short exponential sums} 
As discussed in \S~\ref{sec:Reduction to a counting problem}, one
relates the uniform distribution of the eigenfunctions of the operator
$U_N(A)$, as well as the decay of off-diagonal matrix elements, to
bounding the number of solutions $Q_{2\nu}(N;\vec n)$ for
$\vec n\in \Z^{2g}$ to the matrix congruence~\eqref{eq: Cong with-n},
see  Lemma~\ref{lem:basic inequality new}.

Following the discussion after Theorem~\ref{thm:integers} and Lemma~\ref{lem:basic inequality new}, we thus reduce the problem to showing that 
\[
Q_{2\nu}(p;\vec n)  = o\left(\frac{\ord(A,p)^{4\nu}}{p^g} \right)
\]
for a set of `good' primes $p$ for which the characteristic polynomial of $A$ splits completely over $\F_p$, with  eigenvalues $\lambda_i\in \F_p^*$, $i=1,\ldots,2g$.

 In turn, using the orthogonality of exponential sums, this leads us to a problem of obtaining nontrivial cancellation in exponential sums of the form
\[
\sum_{j=1}^{\ord(A,p)} \ep\( \alpha_1 \lambda_1^{j} + \ldots+ \alpha_{2g}\lambda_{2g}^{j}\)
  \]
  for $(\alpha_1,\ldots,\alpha_{2g})\in \F_p^{2g}$.
  
   These exponential sums are not
 treatable by algebro-geometric methods of Weil and Deligne, but
 fortunately they can be treated by methods from additive
 combinatorics.  In particular, we make use of the bounds of
 Bourgain~\cite[Corollary]{Bourgain Mordell}
 on Mordell type sums over prime  fields.
   
   \begin{lem} \label{lem:exp sum split} 
   For every $\varepsilon > 0$ there exists some $\delta > 0$ such that the following holds. 
Let $\alpha_1, \ldots, \alpha_s \in \F_p$, not all zero, and $\lambda_1,\ldots, \lambda_s\in \F_p^*$ be such that
\[
\ord(\lambda_i,p),\quad  \ord(\lambda_i/\lambda_j,p) \ge p^{\varepsilon}, 
\qquad  1\le i, j\le s, \ i \ne j. 
\]
Then 
\[
 \left|\sum_{x=1}^{T} \ep\( \alpha_1 \lambda_1^{x} + \ldots+ \alpha_{s}\lambda_s^{x}\) \right| \ll T p^{-\delta}, 
  \]
  where $T$ is the order of the  subgroup of $\F_p^*$ generated by $\lambda_1,\ldots, \lambda_s$.
  \end{lem}
  
According to Lemma~\ref{lem:syst eigen}, in the split  case, that is,
for  $\lambda_i\in \F_p^*$, $i=1,\ldots,2g$, the number of solutions
to the system~\eqref{eq:syst eigen} is  given by $Q_{2\nu}(p;\vec
n)=R_{2\nu}(1,\ldots,1;p)$  (with $\newr=2g$ therein).   
 Using the  orthogonality of exponential functions  we obtain
 \[
 \begin{split}
Q_{2\nu}(p;\vec n)&=R_{2\nu}(1,\ldots,1;p) \\
& = \frac 1{p^{2g}} \sum_{\vec \alpha \in \F_p^{2g}} \left| \sum_{t=1}^T \ep\left( \alpha_1 \lambda_1^t+\ldots +\alpha_{2g}\lambda_{2g}^t\right) \right|^{4\nu}
\\
&< \frac{T^{4\nu}}{p^{2g}} + \max_{\0\neq \vec \alpha\in \F_p^{2g}} \left| \sum_{t=1}^T \ep\left( \alpha_1 \lambda_1^t+\ldots +\alpha_{2g}\lambda_{2g}^t\right) \right|^{4\nu} .
 \end{split}
 \]
  Inserting Lemma~\ref{lem:exp sum split}, exactly as in~\cite{BourgainGAFA}, we derive
  \[
Q_{2\nu}(p;\vec n)\ll \frac{T^{4\nu}}{p^{2g}} + T^{4\nu} p^{-4\nu \delta} \leq 2 \frac{T^{4\nu}}{p^{2g}} 
  \]   for $\nu\geq g/(2\delta)$. Hence we find:
 \begin{cor} \label{cor:exp sum split and Q} 
 Let $A\in \Sp(2g,\Z)$ have  separable characteristic polynomial. 
 For every $\varepsilon > 0$ there exists some integer $\nu_0 > 0$  such that the following holds.  For  a prime $p$ so that $A$ splits modulo $p$, let the eigenvalues  of $A$ be 
\[
 \lambda_1,\ldots, \lambda_{2g}\in \F_p^* . 
 \]
Assume that
\[
\ord(\lambda_i,p),\quad  \ord(\lambda_i/\lambda_j,p) \ge p^{\varepsilon}, 
\qquad  1\le i, j\le 2g, \ i \ne j. 
\]
Then,  for  any vector $\vn \in \Z^{2g}$,
which is not a zero-divisor and such that for  $p > p_0(A)
\|\vn\|_2^{2g}$, where $p_0(A)$ is as in Lemma~\ref{lem:syst eigen},
and all $\nu\geq \nu_0$,  
we have
\[
Q_{2\nu}(p;\vec n) \ll \frac{\ord(A,p)^{4\nu}}{p^{2g}}. 
 \]
  \end{cor}
  
 \subsection{Bounding  $\langle \opT_p^{(r)}(\vn)\psi ,\psi'  \rangle $ for a positive proportion of primes}
 
  We remark that the assumptions of Lemma~\ref{lem:exp sum split}  and
Corollary~\ref{cor:exp sum split and Q} 
 hold for a positive
proportion of primes $p$ (in fact, for a full density subset of the
set of primes $p$ for which the characteristic polynomial of $A$
splits completely, see Lemma~\ref{lem:GoodPrimes}).   
Hence, combining Lemma~\ref{lem:basic inequality new} and 
 Corollary~\ref{cor:exp sum split and Q}, we obtain the desired 
 estimate~\eqref{eq:key bound on p intro}  
 on $ \langle \opT_p^{(r)}(\vn) \psi ,\psi'  \rangle$ when $\vn$ is not a zero-divisor. 
 
 \begin{cor}\label{cor:Tpn-bound}
Let $A\in \Sp(2g,\Z)$ have  separable characteristic polynomial. There exists some constant $\gamma>0$, depending only on $A$,  
such  that 
for a positive proportion of primes $p$ the following holds: 
For all  integers $r$ coprime to $p$, 
and for any    $\vn\in \Z^{2g}$, which is not a zero-divisor and with
$p > p_0(A) \|\vn\|_2^{2g}$, where $p_0(A)$ is as in
Lemma~\ref{lem:syst eigen},  
\[
\max_{\psi , \psi'  }
  \left| \langle \opT_p^{(r)}(\vn)\psi ,\psi'  \rangle \right|
  \le  p^{-\gamma},
  \]
  the maximum over all pairs of normalized eigenvectors of  $U_{p,r}(A)$. 
\end{cor}

\section{Treatment of zero-divisors}\label{sec:zero divisors}

\subsection{Preliminaries} We remark that if $f_A$ is irreducible then
there are no nontrivial zero-divisors,  
and thus the results of
\S~\ref{sec:Reduction to a counting problem} allow us to complete the
proof. 
However in the case when $f_A$ is separable but not irreducible we need additional considerations to treat vectors  $\vn\in \Z^{2g}$ which are zero-divisors, as defined in \S~\ref{sec:linear indep}.  Thus this section is not needed if one is only interested in the case of matrices $A \in \Sp(2g,\Z)$ with irreducible characteristic polynomials.

\subsection{Remarks on symplectic spaces}
\label{sec:misc}

We next record some basic facts regarding symplectic vector spaces.
Let $W$ be a symplectic space, that is, a vector space with a
non-degenerate alternating bilinear form, which we denote $ \langle \cdot, \cdot \rangle$.  
 We note that a subspace
$U \subseteq V$ is symplectic, that is, the restriction of the symplectic
form to $U$ is non-degenerate, if and only if
$U \cap U^{\perp} = \{\0 \}$.

\begin{lem}
\label{lem:InvarSubspace}
  Let $A \in \Sp(V)$ be a symplectic matrix over $V$.  Assume that $U \subseteq V$ is an $A$-invariant
  subspace on which $A$ acts irreducibly, and assume that $U$ is not
  isotropic.  Then $U$ is symplectic, and its orthogonal complement
  $U^{\perp}$ is also $A$-invariant and symplectic.
\end{lem}

\begin{proof}
  Assume for contradiction that the restriction of the above bilinear form $ \langle \cdot, \cdot \rangle$ to $U$ is
  degenerate.  
  Then there exist nonzero $\vu_{0} \in U$ such that
  $\langle \vu, \vu_{0}  \rangle = 0$ for all $\vu \in U$, and hence
  $\langle A^{i}\vu, A^{i}\vu_{0}  \rangle = 0$ for all $\vu \in U$ 
  and all 
  integers $i \geq 0$ (note that here we follow the usual 
    convention of groups acting on the left.)

  Since $A$ is symplectic it is invertible, and
  so is the restriction to $U$, hence
  $\langle \vu, A^{i}\vu_{0}  \rangle = 0$ for all $\vu \in U$. Since the
  span of $A^i\vu_{0}$, $i =0,1, \ldots$, equals $U$ we find that
  $U \subseteq U^{\perp}$, contradicting that $U$ is not isotropic.

  The argument for the first part of second assertion is similar. If
  $\vw \in U^{\perp}$ then $\langle \vu,\vw \rangle = 0$ for all $\vu \in U$,
  and thus $\langle A\vu,A\vw \rangle = 0$ for all $\vu \in U$ and hence, again
  using that $A|_{U}$ (that is, the map induced by $A$ on $U$) 
  is invertible, we have
  $\langle \vu,A\vw \rangle = 0$ for all $\vu \in U$ and thus
  $U^{\perp}$ is 
  $A$-invariant.  Since $U$ is symplectic we have
  $U \cap U^{\perp} =\{0\}$ and thus $W = U \oplus U^{\perp}$ (since
  $\dim(U)+\dim(U^{\perp}) = \dim(W)$ always holds). 
  
  Now, if the
  restriction of the form to $U^{\perp}$ is degenerate there exists
  $\vv \in U^{\perp}$ with $\langle \vv, \vu^{\perp} \rangle = 0$ for all
  $\vu^{\perp} \in U^{\perp}$, and since $\langle \vv,\vu\rangle = 0$ for
  all $\vu \in U$, we find that $\langle \vv, \vw \rangle = 0$ for all
  $\vw \in W$, which contradicts $W$ being symplectic.
\end{proof}

A simple consequence of  Lemma~\ref{lem:InvarSubspace}  is that if $W$ splits into
irreducible $A$-invariant subspaces, then each such subspace is either
symplectic or isotropic. If there exist an invariant isotropic
subspace, there is scarring as shown by Kelmer~\cite[Theorem~1]{KelmerAnnals}. 
Otherwise, we can
decompose $W$ into smaller invariant symplectic subspaces and use a
certain tensor product structure to reduce the dimension, and this
allows us to treat the problem of small zero-divisors.

\subsection{Quantized cat maps and tensor products revisited}
\label{sec:CatMap_TensorProd}

Let $A\in \Sp(2g,\Z)$ have separable characteristic polynomial and let
$N=p$ be a prime. Let us consider an element $\vn \in \Z^{2g}$ for
which the reduction modulo $p$ in
$ \Z^{2g}/(p \Z^{2g}) \simeq \Fp^{2g}$ is not a zero-divisor in the
sense defined in \S~\ref{sec:linear indep}, where we identify
$\Fp^{2g} \simeq \Fp[x]/(f_{A}(x))$. 
In order to bound the matrix 
coefficient $ \langle \TNr(\vn)\psi ,\psi' \rangle$ we need some
further properties of the quantization related to invariant symplectic
subspaces and an associated tensor product structure; these properties
are consequences of $U_{p,r}(A)$ being implicitly defined
via the Weil (or oscillator) representation of $\Sp(2g,\F_{p})$. We
 briefly outline the construction below, for more details 
 see~\cite{KelmerAnnals,gurevich-hadani-symplectic}.

Hereafter,  to simplify the notation in this section we  regard $p$
as a fixed prime, and suppress the dependence on $p$ 
and $\vn$ in most places.  
Let
$W$ be a symplectic vector space over $ \F_p$, and assume that $W$
splits into a direct sum of symplectic subspaces, that is,
$W = W_{1} \oplus W_{2}$ where $W_{1} \perp W_{2}$ (that is, $W_2=W_1^{\perp}$), and the
restrictions of the symplectic form to $W_{1}$ and $W_{2}$ are both
non-degenerate.  

We  
emphasise  that in our application, $W_{1},W_{2}$  depend not only
on $p$ but on $\vn$ as well: we write $\Fp^{2g} \simeq W_{1} \oplus
W_{2}$, where the image of $\vn$ in $W_{2}$ is zero, whereas the image
in $W_{1}$ does  not correspond to  a zero-divisor.

With $V_{i} \subseteq W_{i}$, $i=1,2$, 
denoting maximal isotropic subspaces, we
note that $V = V_{1} \oplus V_{2} \subseteq W $ is a maximal isotropic
subspace.  
We may define the Heisenberg group
\[
H(W) = \{ (f,\vw) :~f \in \F_p,\  \vw \in W\}
\] with the group law given by
\[
(f,\vw) \cdot (f',\vw') = (f+f'+\langle \vw, \vw'\rangle, \vw+\vw')
\]
where $ \langle \cdot, \cdot \rangle$ denotes the symplectic form on $W$
(and similarly $H(W_{i})$ for $i=1,2$).

Let $Z \subseteq H(W_{1}) \times H(W_{2})$ denote the subgroup
\[
Z = \{ (f, \0) \times (-f, \0) :~f \in  \F_p \}.
\]
We find that the surjection $H(W_{1}) \times H(W_{2}) \to H(W)$, given
by 
\[
(f_{1},\vw_{1}) \times (f_{2},\vw_{2}) \to (f_{1}+f_{2},
\vw_{1}+\vw_{2})
\]
factors through $Z$, and that we have the isomorphism
\[
(H(W_{1}) \times H(W_{2}))/Z \cong H(W).
\]

The irreducible non-abelian representations of $H(W)$ arise in the
following way.  Given a non-trivial additive character 
$\chi : \F_p \to {\mathbb C}$, let $K = \F_p \times V \subseteq H(W)$
denote a maximal abelian subgroup of $H(W)$ and extend $\chi$ to $K$
(say, by letting $\chi(f,v) = \chi(f)$). 
We remark that the
character $\chi$ depends on $r$ present in the definition of our
observables $T_{p}^{(r)}$, but the precise dependence is not important;
we only need  that $\gcd(r,p) = 1$ implies that $\chi$ is non-trivial. 
By inducing the extended character $\chi$ from $K$ to $H(W)$, we
obtain an irreducible representation $\rho:~H(W) \to GL( L^{2}(V))$,
and similarly irreducible representations
\[
\rho_{\nu} :~H(W_{\nu}) \to GL( L^{2}(V_{\nu})), \qquad \nu=1,2.
\]
 Now, as
$V = V_{1} \times V_{2}$ we have
$L^{2}(V) = L^{2}(V_{1}) \otimes L^{2}(V_{2})$. Since the action of
$Z$ is trivial, we find that $H(W_{1}) \times H(W_{2})$, and thus
$H(W)$, acts in a natural way on $L^{2}(V_{1}) \otimes L^{2}(V_{2})$.

Briefly, the Weil representation $\pi$ of $\Sp(2g, \F_p)=\Sp(W)$ is
then defined as follows: $\Sp(W)$ acts on $H(W)$, and this induces an
action on the set of irreducible representations of $H(W)$. The action
preserves the central character, and since irreducible representations
of $H(W)$ are determined by their central characters (this holds since
$H(W)$ is a two step nilpotent group), the action on the set of
irreducible representations is, up to intertwining operators,
trivial. In particular, for each $g \in \Sp(W)$, define $\rho^{g}$ by
$\rho^{g}(h) = \rho( g(h))$ (for $h \in H(W)$); we then find that
$\rho \simeq \rho^{g}$, that is, there exists an intertwining operator
(only defined up to a scalar; it  turns out that this
gives projective representation of $\Sp(W)$; for $p$ odd a non-trivial
fact is that it is possible to choose scalars to obtain a true
representation) $\pi(g)$ acting on
$L^{2}(V) = L^{2}(V_{1}) \otimes L^{2}(V_{2})$ so that
$\pi(g) \rho^{g} = \rho \pi(g)$.  Further, we similarly obtain
``smaller'' Weil representations $\rho_{\nu}$ of $\Sp(W_{\nu})$ acting
on $L^{2}(V_{\nu})$, for $\nu=1,2$; to fix compatible central characters it is
convenient to use the maps
$H(W_{\nu}) \to H(W_{1}) \times H(W_{2}) \to H(W)$ to obtain the
action of $\Sp(W_\nu)$ on $L^{2}(V_{\nu})$.

The product $\Sp(W_{1}) \times \Sp(W_{2})$,
under the inclusion 
\[\Sp(W_{1}) \times \Sp(W_{2})
\subseteq\Sp(W),
\]
then acts componentwise on the tensor product
$L^{2}(V_{1}) \otimes L^{2}(V_{2})$. In particular, if $A \in\Sp(W)$
leaves both $W_{1}$ and $W_{2}$ invariant, let
$A_{\nu} \in\Sp(W_{\nu})$ denote the corresponding restrictions of $A$
to $W_{\nu}$, for $\nu=1,2$.  We now note that letting
$\vw=\vw_{1}+\vw_{2}$ denote the reduction of $\vn$ modulo $p$, we can
write,  
\begin{equation}
\label{eq:Upr Tpr}
\begin{split}
& U_{p,r}(A) =  U_{1} (A_{1}) \otimes U_{2}(A_{2}), \\
& \Tp^{(r)}(\vn) = \rho((0,\vw)) = \rho_{1}((0,\vw_{1})) \otimes \rho_{2}((0,\vw_{2})),
\end{split}
\end{equation} 
where $U_{p,r}(A) = \pi(A)$ and $U_{\nu}(A_{\nu}) =
\pi_{\nu}(A_{\nu})$ for $\nu=1,2$.

\subsection{Eigenfunctions of tensor products}
\label{sec:eigenf-tens-prod}
We next describe eigenfunctions of $U_{p,r}(A)$
in terms of the tensor product structure.  With $W,W_1,W_2$ and $V,V_{1},V_{2}$ as in
\S~\ref{sec:CatMap_TensorProd}, for $\nu=1,2$, we may decompose
$L^{2}(V_{\nu})$ into $U_{\nu}(A_{\nu})$-eigenspaces
\[
E_{\nu,\lambda} = \ker( U_{\nu}(A_{\nu}) - \lambda I), \qquad 
\lambda \in \Lambda_{\nu},
\]
(possibly with
multiplicities), where $\lambda$ ranges over the set of eigenvalues
$\Lambda_{\nu}$ of $U_{\nu}(A_{\nu})$. 

Further, for $\nu=1,2$ we may find bases of orthonormal eigenfunctions
$\psi_{\nu,\lambda,i} \in E_{\nu,\lambda}$, $i = 1, \ldots, I_{\nu,\lambda}$, 
for some   positive integers $ I_{\nu,\lambda}= \dim(E_{\nu,\lambda})$ with 
\[
\sum_{\lambda \in \Lambda_{1}}  I_{1,\lambda} + \sum_{\lambda \in \Lambda_{2}}  I_{2,\lambda}=  2g,
\]
which follows from the separability of the characteristic polynomial of $A$.
That is,
\[
U_{\nu}(A_{\nu}) \psi_{\nu,\lambda,i} = \lambda_{\nu,i} \psi_{\nu,\lambda,i}, \qquad 
\nu = 1,2, \  i = 1, \ldots, I_{\nu,\lambda}, 
\]
where
$\lambda_{\nu,i}$ ranges over the whole set $\Lambda_{\nu}$. 
We further note that the set 
\[
\{ \psi_{1,\lambda_{1},i_1} \otimes \psi_{2,\lambda_{2},i_2}:~
\lambda_{\nu} \in \Lambda_\nu, \ i_\nu =  1, \ldots, I_{\nu,\lambda}, \
\nu =1,2\}
\]
 gives an orthonormal eigenbasis of
$L^{2}(V) = L^{2}(V_{1}) \otimes L^{2}(V_{2})$.
In particular, the eigenvalues
of $U_{p,r}(A) = U_{1}(A_{1}) \otimes
U_{2}(A_{2})$ are given by
\[
\Lambda = \{ \lambda_{1} \lambda_{2} :~\lambda_{1} \in \Lambda_{1}, \lambda_{2} \in \Lambda_{2}\},
\] 
and for
$\mu \in \Lambda$, an eigenbasis for $E_{\mu} = \ker( U_{p,r}(A) - \mu I)$
is given by 
\[
\{ \psi_{1,\lambda,i} \otimes \psi_{2,\mu/\lambda,j}:~
\lambda \in \Lambda_1, \ 
i =  1, \ldots, I_{1,\lambda}, \ j =  1, \ldots, I_{2,\mu/\lambda}
\}.
\]
Note that the quantizations
  $U_{p,r}(A), U_{1}(A_{1})$, and $U_{2}(A_{2})$ are only defined up
  to scalars, but once we have chosen scalars for $U_{1}(A_{1})$, and
  $U_{2}(A_{2})$ we may chose the scalar for $U_{p,r}(A)$ so that
  multiplicativity of eigenvalues hold.

We can now bound matrix coefficients corresponding to observables
having zero-divisors. 

\begin{lem} 
\label{lem:Bound ZeroDiv}
Let $A\in \Sp(2g,\Z)$ with a  separable characteristic polynomial, such that there are no $A$-invariant rational istropic subspaces.
 There exists some constant $\gamma>0$, depending only on $A$,   
such  that 
for a positive proportion of primes $p$ the following holds: Let $\psi \in E_{\mu}$ and $\psi' \in E_{\mu'}$   denote two eigenfunctions of $U_{p,r}(A) $, and let $\vw$ denote a
non-trivial zero-divisor.  
Then for $p > p_0(A) \|\vw\|_2^{2g}$, where $p_0(A)$ is as in
Lemma~\ref{lem:syst eigen}, we have 
\[
|\langle \Tp^{(r)}(\vw) \psi,\psi'\rangle| 
\ll
p^{-\gamma} \lVert \psi \rVert_{2} \cdot \lVert \psi' \rVert_{2}.
\]
\end{lem}

\begin{proof} 
Let $\0\neq \vw\in \Z^{2g}$, which is a zero-divisor. Then there is an $A$-stable rational subspace $W_1$, necessarily symplectic by Lemma~\ref{lem:InvarSubspace}, so that with respect to the decomposition $\vw= (\vw_{1},\vw_{2}) \in W_1\oplus W_2$, where $W_2= W_1^\perp$, the component $\vw_1\in W_1$ of $\vw$ is  not a zero divisor, while the component $\vw_{2}$ in $W_2=W_1^\perp$ is zero. 

For $\mu, \mu' \in \Lambda$, write
\[
\psi = \sum_{(\lambda, i, j) \in \Omega} \alpha_{\lambda,i,j}
\psi_{1,\lambda,i} \otimes \psi_{2,\mu/\lambda,j},
\]
where
\[
\Omega = \{(\lambda, i, j):~
\lambda \in \Lambda_1, \ i =  1, \ldots, I_{1,\lambda}, \ j =  1, \ldots, I_{2,\mu/\lambda}\},
\]
and
\[
\psi' = \sum_{(\lambda', i', j')\in \Omega'}  \beta_{\lambda',i',j'}
\psi_{1,\lambda',i'} \otimes \psi_{2,\mu'/\lambda',j'},
\]
where 
\[
\Omega' = \{(\lambda', i', j'):~
\lambda' \in \Lambda_1, \ i' =  1, \ldots, I_{2,\lambda'}, \ j' =  2, \ldots, I_{2,\mu'/\lambda'}\},
\]
with complex coefficients  $\alpha_{\lambda,i,j}, \beta_{\lambda',i',j'} \in {\mathbb C}$.

Since $\vw_{2}=\0$, by~\eqref{eq:Upr Tpr}, 
we have 
\[
 \Tp^{(r)}(\vw)=\rho((0,\vw))= \rho_{1}((0,\vw_{1})) \otimes \rho_{2}((0,\vw_{2})) =
\rho_{1}((0,\vw_{1})) \otimes \mathrm{Id},
\] 
and thus
\begin{align*}
\langle \rho((0,\vw)) \psi, \psi'\rangle
= \sum_{(\lambda, i, j) \in \Omega} 
   \sum_{(\lambda', i', j')\in \Omega'} & \overline{\alpha_{\lambda,i,j}}
\beta_{\lambda',i',j'}  
\\
\cdot \langle \rho_1((0,\vw_1)) &\psi_{1,\lambda,i},\psi_{1,\lambda',i'}   \rangle 
 \langle \psi_{2,\mu/\lambda,j}, \psi_{2,\mu'/\lambda',j'} \rangle.
\end{align*}
Now, since 
\[
\langle \psi_{2,\mu/\lambda,j}, \psi_{2,\mu'/\lambda',j'}
\rangle =  
\begin{cases} 1&\text{if  $j=j'$ and  $\mu/\lambda = \mu'/\lambda'$,}\\
0&\text{otherwise,}
\end{cases} 
\]
only terms for which $j=j'$ and for which $\lambda' = \eta(\lambda)$
for the  bijection $\eta : \Lambda_{1} \to \Lambda_{1}$
contribute (more precisely, we have
$\eta(\lambda) =  (\lambda \mu')/\mu$. 
Hence 
\begin{equation}\label{eq:rhov and rhov1}
\begin{split}
& \langle\rho((0,\vw)) \psi, \psi'\rangle\\
&\qquad \quad =
\sum_{(\lambda,i,j) \in \Omega} \sum_{i'=1}^{I_{1, \eta(\lambda)}}
\overline{\alpha_{\lambda,i,j}}
\beta_{\eta(\lambda),i',j}
\langle \rho_1((0,\vw_1)) \psi_{1,\lambda,i},\psi_{1,\eta(\lambda),i'}    \rangle.
\end{split}
\end{equation}

We now apply Corollary~\ref{cor:Tpn-bound} with respect to the matrix $A_1$ in 
the decomposition~\eqref{eq:Upr Tpr}, which applies  since $\vw_{1}$ is not a 
zero-divisor.
Then,  by the Cauchy inequality,  for every $\lambda$ and $j$ fixed, we have 
\begin{equation}\label{eq:Bound-prelim}
\begin{split}
& \left| 
  \sum_{i=1}^{I_{1, \lambda}} \sum_{i'=1}^{I_{1, \eta(\lambda)}}
\overline{\alpha_{\lambda,i,j}}
\beta_{\eta(\lambda),i',j}
\langle \rho_1((0,\vw_1)) \psi_{1,\lambda,i},\psi_{1,\eta(\lambda),i'}    \rangle
\right|\\
& \qquad \qquad \ll
p^{-\gamma}
\left( \sum_{i=1}^{I_{1, \lambda}}|\alpha_{\lambda,i,j}|^{2}   \right)^{1/2}
\left( \sum_{i'=1}^{I_{1, \eta(\lambda)}}
|\beta_{\eta(\lambda),i',j}|^{2}\right)^{1/2}
\end{split}
\end{equation}

Finally, using the Cauchy inequality  again, and then recalling that
$\eta$ is a bijection on $\Lambda$, we derive 
\begin{align*}
\sum_{\lambda \in \Lambda} & \sum_{j=1}^{I_{2, \mu/\lambda}} 
\left( \sum_{i=1}^{I_{1, \lambda}}|\alpha_{\lambda,i,j}|^{2}   \right)^{1/2} 
\left( \sum_{i'=1}^{I_{1, \eta(\lambda)}}
|\beta_{\eta(\lambda),i',j}|^{2}\right)^{1/2}
\\
&  \ll
\(\sum_{\lambda \in \Lambda} \sum_{j=1}^{I_{2, \mu/\lambda}}
\sum_{i=1}^{I_{1, \lambda}}|\alpha_{\lambda,i,j}|^{2}  \)^{1/2} 
 \(\sum_{\lambda \in \Lambda} \sum_{j=1}^{I_{2, \mu/\lambda}}
 \sum_{i'=1}^{I_{1, \eta(\lambda)}}
|\beta_{\eta(\lambda),i',j}|^{2} \)^{1/2} 
 \\ &  \qquad \qquad \qquad \qquad \qquad \qquad \qquad \qquad \qquad \qquad =  \lVert \psi \rVert_{2}
\cdot
\lVert \psi' \rVert_{2}
\end{align*} 
and recalling~\eqref{eq:rhov and rhov1} and~\eqref{eq:Bound-prelim},  we conclude the proof. 
\end{proof}

\section{Anatomy of integers}\label{sec:anatomy}
\subsection{Some sums and products over primes} 
It is convenient to denote by  $\log_{k} x$ the $k$-fold iterated logarithm,
that is,  for $x\ge 1$ we set 
\[
\log_1 x =  \log x \mand  \log_k = \log_{k-1} \max\{\log x, 2\}, \quad k =2, 3, \ldots.
\]
We begin by recording an upper bound for Mertens type sums over primes
in progressions, together with a simple consequence.

\begin{lem}
  \label{lem:mertens-for-progression}
  Let $q$ be a prime and let $j \ge 1$ be an integer.
  We have
  \[
  \sum_{\substack{p \le x\\q \mid p^{j} -1} } \frac{1}{p}
  \ll
q^{-1/j} + \frac{\log_2 x }{q},
  \] 
  where the implied constant depends only on $j$. 
\end{lem}

\begin{proof}
  For  an integer $k \geq 0$ define the dyadic interval $I_{k} = [ 2^{k}q, 2^{k+1}q ]$, and
  note that $q\mid p^{j}-1$ implies that $p$ must lie in a progression
  $p \equiv a \mod q$, where $0 \le a < q$ ranges over over at most
  $j$ possible values.
  For any $a$, the Brun--Titchmarsh inequality, see, for example,~\cite[Theorem~6.6]{IwKow} or~\cite[Chapter~I, Theorem~4.16]{Ten}, implies that
\[
\sum_{\substack{p \in I_{k} \\ p \equiv a \mod q} } 1/p
\ll \frac{2^{k+1}q }{ q \log( 2^{k+1}q/q)   }    \cdot \frac{1}{2^{k}q}
\ll \frac{1}{q (k+1)}.
\]
If $2^{k}q \le x$ we have $k \ll \log x$, and summing over such $k$
we find that the contribution from primes $p \ge q$ is $O\(q^{-1} \log_{2}x\)$. Since there are at most $j$ primes $p < q$ for which
$q \mid p^{j}-1$, and each such prime satisfies $p > q^{1/j}$ we find
that the contribution from $p < q$ is $O\(q^{-1/j}\)$, and the
proof is concluded.
\end{proof}

We remark that for $j=1$ the bound of Lemma~\ref{lem:mertens-for-progression}
simplifies as 
\begin{equation}
  \label{eq:Mertens Progr}
  \sum_{\substack{p \le x\\ p \equiv 1 \bmod q } } \frac{1}{p}
  \ll \frac{\log_2 x }{q}. 
\end{equation}

We control the contribution from small prime divisors of $p -1$ as
follows.  For a prime $q$ and positive integer $k$, we define $v_q(k)$ to be the positive integer $\ell$ such that
\[
q^\ell \mid k \mand q^{\ell+1} \nmid k.
\]

We fix some $z > 0$ and  let
\begin{equation}
\label{eq:sz}
s_z(N) =  \prod_{p \mid N} \prod_{q \le z} q^{v_q(p-1)}
= \prod_{p \mid N} \prod_{\substack{q \le z\\ q^{\ell} \| p -1}} q^{\ell},
\end{equation} 
that is, $s_z(N)$ is the product of the $z$-smooth parts of $p-1$, as $p$
ranges over all prime divisors of $N$.

\begin{lem}
\label{lem:smooth part}
Let 
\[
Z = \exp\((\log_{2} x) (\log_{3} x)^{3/2}\) \mand z = \(\log_2 x\)^{O(1)}.
 \]  
 For all but $o(x)$ integers $N \le x$ we have $s_z(N) \le Z$.
\end{lem}

\begin{proof} From the definition of $s_z(N)$ in~\eqref{eq:sz}, extending over all powers $q^\ell\le x$, $q\le z$, such that $q^\ell\mid (p-1)$, we have
\[
\sum_{N \le x}  \log s_z(N) \ll
\sum_{\substack{q^{\ell} \le x\\ q \le z,~\mathrm{prime}}}
\log(q^{\ell})
\sum_{p  \equiv 1 \mod q^{\ell}} \fl{x/p}= S_1 + S_{\ge 2},
\]
where $S_1$ is the contribution from the terms corresponding to  $\ell=1$ 
and $S_{\ge 2}$ is the contribution from the terms with $\ell \ge 2$. 

For $S_1$, we have
\[
S_1 
\ll x
\sum_{q \le z,~\mathrm{prime} } \log q \sum_{\substack{p\le x\\ p  \equiv 1 \mod q}} \frac{1}{p}.
\]  
Using~\eqref{eq:Mertens Progr} applied to the inner sum, we now derive 
\begin{equation}\label{eq:S1}
\begin{split} S_1 & \ll x
\sum_{q \le z} \log q \frac{\log_{2} x}{q} \\
&  \ll x (\log_2 x)  \sum_{q \le z}   \frac{\log q}{q}
\ll   x (\log_2 x) (\log z).
\end{split}
\end{equation}

The sum  $S_{\ge 2}$ is estimated trivially by discarding the primality 
conditions on $p$ and thus using that 
\[
\sum_{\substack{p\le x\\ p  \equiv 1 \mod q^{\ell}}} \frac{1}{p} \le \sum_{1 \le k \le  x/q^{\ell}} \frac{1}{1+k q^\ell} 
\ll  \frac{\log x}{q^\ell} ,
\] 
which implies, after we abandon the condition of primality on $q$ and the inequality $q \le z$,
\begin{equation}\label{eq:S2}
\begin{split}
S_{\ge 2}&  \ll x(\log x) \sum_{2 \le \ell\le \log x/\log 2} \, \sum_{1\le m  \le x^{1/\ell}}
  \frac{\log(m^{\ell})}{m^\ell}  \\
 & \ll x (\log x)^2 \sum_{2 \le \ell\le \log x/\log 2}   x^{-1 +1/\ell}\ll x^{1/2}  (\log x)^2.
\end{split}
\end{equation}

Clearly the bound on $S_1$ in~\eqref{eq:S1}  dominates  the bound on $S_{\ge 2}$ in~\eqref{eq:S2}.
Hence,  
\[
\sum_{N \le x}  \log s_z(N) \ll  x (\log_2 x) (\log z) \ll x (\log_2 x) (\log_3 x).
\]
Therefore we have
$s_z(N) \ge  Z = \exp\((\log_{2} x) (\log_{3} x)^{3/2}\) $ 
 for at most  
 \[
 O\(x (\log_2 x) (\log_3 x) (\log Z)^{-1} \) = O\(x\(\log_3 x\)^{-1/2}\)
 \]  positive integers $N \le x$.
\end{proof}

\subsection{Good primes and integers} 
We recall that $A \in \Sp(2g,\Z)$.

  We say that a prime $p$ is {\it good\/} if the following two conditions are satisfied:
  \begin{itemize}
\item  the
characteristic polynomial of $A$ is separable and  splits completely modulo $p$;
\item  for the roots  $\lambda_1, \ldots, \lambda_{2g}$ of    the
characteristic polynomial of  $A$ modulo $p$ we have
\[
\ord(\lambda_i,p),\,  \ord(\lambda_i/\lambda_j,p) \ge p^{1/3}, 
\qquad  1\le i, j\le s, \ i \ne j. 
\]
\end{itemize}

We note that the exponent $1/3$ is somewhat arbitrary and can be replaced by any $\gamma < 1/2$. 

Let $\cPg$ denote the set of
good primes. 

Applying  Lemma~\ref{lem:ord},  we  now derive

\begin{lem}
  \label{lem:GoodPrimes}
The set $\cPg$  is of positive density.  
\end{lem}

Next,  given integers $U \ge V \ge 1$ we define 
\[
\cPg(V,U) = \cPg\cap [V,U].
\] 

We now set 
\begin{equation}\label{eq:UVW}
\begin{split}
 & D(x) =  ( \log x)^{(\log_{3} x)^{2}},\\
&  V(x) =  \exp \(\exp\(\sqrt{\log_2 x}\)\),  \\
& W(x) = x^{\log_{3}x/\log_{2} x} , \\
\end{split} 
\end{equation} 
 and define the following set $\cNg(x)$ of
{\it good\/} integers
\begin{equation}\label{eq:good N}
\begin{split}
\cNg = \{ N :~\exists   p&\in  \cPg(V(N),W(N)) \\
& \textrm{with } N = p M, \ M \in \Z,\ \gcd(p, M) =1,  \\
&\qquad \qquad    
 \gcd\(p-1,\ord(A,M)\) \le  D(N)\}. 
\end{split} 
\end{equation}
We then set 
\[
\cNg(x) = \cNg \cap [1,x].
\]

The next statement is our main tool.

\begin{lem}
\label{lem:smallgcd}  
  We have
\[
  \sharp \cNg(x) = x + o(x).
\]
\end{lem}

\begin{proof}  It is certainly enough to show that 
\[\sharp \(\cNg \cap [x/2,x]\) = x/2 + o(x).
\]
In turn, we  set 
\[
 D_0 = D(x/2),   \quad V_0 = V(x),  \quad W_0 = W(x/2) , 
\]  
such that 
\[
[V_0,W_0]\subseteq [V(N),W(N)],
\]  
for all $N \in [x/2, x]$
and define the following set $\tcNg(x)$ of {\it good\/} integers $N \le x$:
\begin{align*}
\tcNg(x) = \{ N \le x &:~\exists  p\in  \cPg(V_0,W_0) \textrm{ with }N = p M, \ M \in \Z, \\
&\quad  \gcd(p, M) =1, \
 \gcd\(p-1,\ord(A,M)\)  \le  D_0\}. 
\end{align*} 
Clearly 
\[
\tcNg(x)  \cap [x/2,x] \subseteq  \cNg \cap [x/2,x],
\] 
hence 
it is enough to show that 
\begin{equation}\label{eq:tilde-N}
  \sharp\tcNg(x) = x + o(x).
\end{equation}  
That is, in the above,  we first consider integers $N$ in a dyadic interval. This allows us to replace 
$ \cPg(V(N),W(N))$ with  $\cPg(V_0,W_0)$. After this is done,  we can bring back integers below $x/2$ as well: if the exceptional set 
is of size $o(x)$ on $[1,x]$  then so it is on $[x/2,x]$ and we are done. 
Thus indeed we only need to establish~\eqref{eq:tilde-N}.

 First recall that by Lemma~\ref{lem:GoodPrimes}  the set of good
 primes $\cP_{good}$   is of positive density. Therefore, there are 
 some constants $C, c>0$ (depending on the matrix $A$)  such that for $Z \ge 2$ the set  $\cPg(Z, CZ)$ contains 
 at least $cZ/\log Z + O(1)$ primes, that is, 
 \begin{equation}\label{eq: Card P(Z, CZ)}
  \sharp \cPg(Z, CZ) \ge c\frac{Z}{\log Z} + O(1).
\end{equation}

Taking $x$ sufficiently large such that the interval $[2,W_0]$ contains $I$ non-overlapping  intervals of the form
$[C^i, C^{i+1} )$, $i =1, \ldots, I$, where
\[
\log W_0 \ll I \ll \log W_0, 
\] we  derive
\begin{align*}
  \sum_{p \in \cPg\(V_0,W_0\)} 1/p  & \ge   \sum_{p \in \cPg(2,W_0)} 1/p -   \sum_{p \le V_0,~\mathrm{prime}}1/p  \\
  & \ge  \sum_{i=1}^I  \sum_{p \in \cPg\(C^i , C^{i+1}\)} 1/p -   \sum_{p \le V_0,~\mathrm{prime}}1/p  \\
    & \ge  \sum_{i=1}^I  C^{-i} \sharp \cPg\(C^i , C^{i+1}\)  -   \sum_{p \le V_0,~\mathrm{prime}}1/p . 
    \end{align*}

Next, recalling~\eqref{eq: Card P(Z, CZ)}  and the Mertens formula (or 
simply using~\eqref{eq:Mertens Progr}  with $q =1$), we obtain 
\begin{align*}
 \sum_{p \in \cPg\(V_0,W_0\)} 1/p  &  
\ge  \sum_{i=1}^I  C^{-i} \(c\frac{C^i}{i \log C} + O(1)\) + O(\log_2 V_0) \\
 & \ge  \frac{c}{\log C}   \sum_{i=1}^I \frac{1}{i}  + O(\log_2 V_0)  \ge  \frac{c}{\log C}  \log I+ O(\log_2 V_0) \\
  & \gg \log_{2} W_0 + O(\log_2 V_0)  \gg  \log_{2} W_0 
 \gg \log_{2} x. 
\end{align*}
Therefore,
\[
\prod_{p \in \cPg\(V_0,W_0\)} (1-1/p) \ll \exp\left( - \  \sum_{p \in \cPg\(V_0,W_0\)} 1/p   \right) \le 
(\log x)^{-\gamma}
\]
for some $\gamma>0$, which depends only on $C$ and $c$, and thus only on the matrix $A$. 
 Thus, by  the classical Brun sieve, see, for example,~\cite[Chapter~I, Theorem~4.4]{Ten}, 
 almost all $N \le x$ are divisible by some prime 
$p \in \cPg\(V_0,W_0\)$.

We now set $z = \(\log_2 x\)^{2g+1}$ and note that $D_0 >  Z$, where $Z$ is as 
in Lemma~\ref{lem:smooth part}.  Thus  Lemma~\ref{lem:smooth part}   allows us to 
discard $o(x)$ positive integers $N \le x$ with 
\[
s_z(N) \ge  D_0,
\] 
where $s_z(N)$ is defined by~\eqref{eq:sz}. 
Hence for the remaining integers $N \in [x/2,x]$ we have 
\[
s_z(N) < D_0 \le D(N).
\]

We also discard $O(x/V_0)$ integers $N\le x$ which are divisible by $p^2$ 
for some prime $p > V_0$. Hence, for the remaining integers $N$, 
for any  $p\in \cPg(V_0,W_0)$  with $p \mid N$ we now have $\gcd(p, N/p) =1$.

Furthermore, for the remaining  $N \le x$, 
we see that if  
\[ 
\gcd\(p-1,\ord(A,N/p)\) >D_0,
\]  then, since $s_z(p) <  s_z(N) < D_0$, 
there is a prime $q> z$ with  $q\mid p-1$ and another prime $\ell \mid N$, $\ell \ne p$,
such that 
\[
q \mid \ord(A,\ell) \mid  \prod_{j=1}^{2g} \(\ell^{j}-1\). 
\]

Hence to conclude the proof   it suffices to show that for every $j =1, \ldots, 2g$ we have 
\begin{equation}\label{eq:Bad N}
  \sum_{\substack{q > z,  \ \mathrm{prime}}}\ 
  \sum_{\substack{ p \leq x,  \ \mathrm{prime} \\ p \equiv 1 \mod q}}\
  \sum_{\substack{ \ell \leq x/p,  \ \mathrm{prime}\\p \ne \ell\\ q \mid  \ell^j-1}}
  \frac{x}{\ell p}
  = o(x). 
\end{equation}

To establish~\eqref{eq:Bad N}, we first discard the condition $\ell  \ne p$, and extend the summation over $\ell$ up to $\ell\le x$.
Then we  recall  Lemma~\ref{lem:mertens-for-progression} (for the sum over $\ell$) and its special case~\eqref{eq:Mertens Progr}  
 (for the sum over $p$)  and derive 
\begin{align*}
  \sum_{q > z, \ \mathrm{prime}}\ 
  \sum_{\substack{ p \leq x,  \ \mathrm{prime}\\ p  \equiv 1 \mod q}} \
  \sum_{\substack{ \ell \leq x/p,  \ \mathrm{prime} \\p  \ne \ell\\ q \mid  \ell^j-1}}
  \frac{x}{\ell p} 
  & \le x  \sum_{q > z,\ \mathrm{prime}}\  
  \sum_{\substack{ p \leq x,  \ \mathrm{prime} \\ p \equiv 1 \mod q}}\frac{1}{p} \
  \sum_{\substack{\ell \leq x,  \ \mathrm{prime}\\ q \mid  \ell^j-1}}
  \frac{1}{\ell} \\
  &  \ll  x \sum_{q>z,  \ \mathrm{prime} }   \frac{\log_2 x }{q} \(q^{-1/j} + \frac{\log_2 x }{q}\)\\
    & \ll  x\( \frac{\log_2 x }{z^{1/j}} + \frac{(\log_2 x)^2 }{z^2}\) \\
       & \ll  x\( \frac{\log_2 x }{z^{1/(2g)}} + \frac{(\log_2 x)^2 }{z^2}\) . 
\end{align*}
Recalling our choice $z = \(\log_2 x\)^{2g+1}$, we  obtain~\eqref{eq:Bad N}.

Thus all together we have discarded $o(x)$ integers and all remaining integers $N\le x$ belong to $\tcNg$. 
Hence we see that~\eqref{eq:tilde-N} holds,  and the result follows. 
\end{proof}

\section{Proof of Theorem~\ref{thm:integers}}

We recall the definition of good integers given by~\eqref{eq:good N}.
We now show that~\eqref{QUE for all N's} holds with $\cN = \cNg$,  that is, 
\[
\lim_{\substack{N\to \infty\\ N\in\cNg}}\max_{\psi_N,\psi'_N }
\left| \langle \OPN(f)\psi_{ N}, \psi_{N}'  \rangle -\langle \psi_N,\psi_N'\rangle \int_{\TT^{2g}}
  f(\vx)d\vx  \right| = 0,
\] 
where the maximum is taken over all pairs of normalized eigenfunctions $\psi_N,\psi'_N$ of $U_N(A)$. 
By Lemma~\ref{lem:smallgcd},  the set $\cNg$ is of full density and hence 
this is sufficient for our goal. 

As in~\cite{BourgainGAFA, KR2001}, using the rapid decay of coefficients of 
 $f\in C^\infty(\TT^{2g})$, it suffices to show that
 \[
\max_{\substack{\vn   \in \Z^{2g}\\
    0 < |\vn| \le L(N)}}
\max_{ \psi_N,\psi_N' }
\left| \langle\TN(\vn) \psi_N , \psi_N'  \rangle\right| \to 0
\]
as $N \to \infty$, $N\in\cNg$, 
with $\psi_N, \psi_N'$ running over all normalized eigenfunctions of $U_N(A)$, with a slowly growing function $L(N)\to \infty$. 

We  recall the definition of the functions $D(x)$, $V(x)$ and $W(x)$ as in~\eqref{eq:UVW}. 
In particular, we take $L(N)$  to grow sufficiently slowly to
guarantee that for 
any $p \in \cPg(V(N),W(N))$ and  
for any $\vn  \in \Z^{2g}$ with  $0 < |\vn| \le L(N)$ the conditions
of Corollary~\ref{cor:Tpn-bound}  and Lemma~\ref{lem:Bound ZeroDiv}
are satisfied provided that $N$ is sufficiently large.

We now  fix some $N \in \cNg$ and choose a prime $p$ which satisfies all properties in~\eqref{eq:good N}.

We set 
\[d = \gcd\(\ord(A,p), \ord(A, M )\).\]
Clearly
\[
d \le \gcd (p - 1, \ord(A, M )) \le D(N).
\] 
Now, applying~\eqref{eq: A->A^d}, and then
Lemma~\ref{lem:tensor-prod-bound} (with $N_1 = p$ and with $A^d$
instead of $A$), we derive
\begin{equation}\label{eq:TN vs Tp}
  |\langle \TN(\vn)\psi, \psi' \rangle|  \le  \max_{\varphi, \varphi' \in \Phi_{p, r}}
  | \langle  \opT_{p}^{ r} (\vn)\varphi, \varphi' \rangle |, 
\end{equation}  where $r$ is some integer coprime to $p$ and $\varphi$, $\varphi'$ range over all normalized eigenfunctions of  $U_{p, r}(A^d)$.  
  
 We note that  the roots  of the   characteristic polynomial of $A^d$ 
 are  $\lambda_1^d, \ldots, \lambda^d_{2g}$, where 
 $\lambda_1, \ldots, \lambda_{2g}$  are the roots  of the   characteristic polynomial of $A$ modulo $p$ 
 and we also have
\[
\ord(\lambda_i^d,p),\,  \ord(\lambda_i^d/\lambda_j^d,p) \ge p^{1/3} d^{-1} \gg p^{1/4}, 
\qquad  1\le i, j\le 2g, \ i \ne j,
\]
since obviously for a sufficiently large $N$ we have 
\[
d\le D(N) \le V(N)^{1/12} \le p^{1/12} .
\]
 Thus the conditions of Corollary~\ref{cor:exp sum split and Q} are satisfied. Combining  Corollary~\ref{cor:Tpn-bound} and Lemma~\ref{lem:Bound ZeroDiv} (when $\vn$ is a zero divisor) with~\eqref{eq:TN vs Tp}
 we conclude the proof.

\end{document}